\documentclass[a4paper,10pt]{article}
 
\usepackage[utf8]{inputenc}
\usepackage{amssymb} 
\usepackage{amsmath}
\usepackage{amsthm}
\usepackage{mathrsfs}
\usepackage{mathtools}
\usepackage{graphicx}
\usepackage{subfigure}
\usepackage{bm}
\usepackage{enumitem}
\usepackage{wasysym}

\theoremstyle{plain} 
\newtheorem{theorem}{Theorem}
\newtheorem*{theorem*}{Theorem}
\newtheorem{proposition}[theorem]{Proposition}
\newtheorem{corollary}[theorem]{Corollary}

\newtheorem{lemma}[theorem]{Lemma}
\newtheorem*{conclusion}{Conclusion}
\theoremstyle{definition}
\newtheorem{remark}[theorem]{Remark}
\newtheorem{example}[theorem]{Example}
\newtheorem{definition}[theorem]{Definition}

\def\mbb #1{\mathbb{#1}}

\def\mbf #1{\mathbf{#1}}
\def\mrm #1{\mathrm{#1}}
\def\Cal #1{\mathcal{#1}}
\def\mfr #1{\mathfrak{#1}}
\def\mscr #1{\mathscr{#1}}

\def\C{\mbb{C}}
\def\CP{\mbb{CP}}
\def\R{\mbb{R}}
\def\Z{\mbb{Z}}
\def\N{\mbb{N}}
\def\K{\mbb{K}}

\def\SL{\mrm{SL}}

\def\tr{\mrm{tr}}
\def\id{\mrm{id}}

\newcommand{\sminus}{\smallsetminus}

\def\sX{\mathsf{X}}
\def\sE{\mathsf{E}}
\def\sY{\mathsf{Y}}
\def\sU{\mathsf{U}}
\def\sXE{\mathsf{XE}}

\makeatletter
\newcommand*\bigcdot{\mathpalette\bigcdot@{.5}}
\newcommand*\bigcdot@[2]{\mathbin{\vcenter{\hbox{\scalebox{#2}{$\m@th#1\bullet$}}}}}
\makeatother

\newcommand{\uppie}{\scalebox{0.7}{\!$\mathrlap{\rotatebox[origin=c]{-45}{$\LEFTCIRCLE$}}\rotatebox[origin=c]{-135}{$\LEFTCIRCLE$}\!$}}
\newcommand{\downpie}{\scalebox{0.7}{\!$\mathrlap{\rotatebox[origin=c]{45}{$\LEFTCIRCLE$}}\rotatebox[origin=c]{135}{$\LEFTCIRCLE$}\!$}}

\newcommand*{\transp}[2][-3mu]{\ensuremath{\mskip1mu\prescript{\smash{\mathrm t\mkern#1}}{}{\mathstrut#2}}}%
\newcommand{\barbm}[1]{\overline{#1}}

\author{Martin Klime\v{s}}
\title{Stokes phenomenon and confluence in non-autonomous Hamiltonian systems}

\begin{document}
\maketitle

\begin{abstract}
\noindent
This article studies a confluence of a pair of regular singular points to an irregular one in a generic family of time-dependent Hamiltonian systems in dimension 2. 
This is a general setting for the understanding of the degeneration of the sixth Painlevé equation to the fifth one.
The main result is a theorem of sectoral normalization of the family to an integrable formal normal form, 
through which is explained the relation between the local monodromy operators at the two regular singularities and the non-linear Stokes phenomenon at the irregular singularity of the limit system. 
The problem of analytic classification is also addressed.

\def\keywords#1{\small{\textbf{Key words:} #1}}\def\and{\ifhmode\unskip\nobreak\fi\ $\cdot$\ }
\smallskip\noindent
\end{abstract}

\section{Introduction}
We consider a parametric family of non-autonomous Hamiltonian systems of the form
\begin{equation}\label{eq:HC-system}
\begin{aligned}
x(x-\epsilon)\tfrac{dy_1}{dx}&=\ \tfrac{\partial H}{\partial y_2}(y,x,\epsilon)\\
x(x-\epsilon)\tfrac{dy_2}{dx}&=-\tfrac{\partial H}{\partial y_1}(y,x,\epsilon),
\end{aligned}
\qquad (y,x,\epsilon)\in(\C^2\!\times\!\C\!\times\!\C,0),
\end{equation}
shortly written as
\begin{equation}\label{eq:HC-systemshort}
x(x-\epsilon)\tfrac{dy}{dx}=J\transp{(D_yH)},\qquad J=\left(\begin{smallmatrix} 0 & 1 \\[3pt] -1 & 0 \end{smallmatrix}\right),
\end{equation}
with a singular Hamiltonian function $\frac{H(y,x,\epsilon)}{x(x-\epsilon)}$,
where $H(y,x,\epsilon)$ is an analytic germ such that $H(y,0,0)$ has a non-degenerate critical point (Morse point) at $y=0$:
\begin{equation*}
 D_yH(0,0,0)=0,\qquad \det D_y^2 H(0,0,0)\neq 0.
\end{equation*}
The last condition means that the $y$-linear terms of the right side of \eqref{eq:HC-systemshort} are of the form 
$$A(x,\epsilon)y\ \text{for}\ A=J D_y^2H,\ \text{where}\ 
A(0,0)\sim \left(\begin{smallmatrix} \lambda^{(0)}(0)\!\!\! & 0 \\[3pt] 0 & \!\!\!-\lambda^{(0)}(0) \end{smallmatrix}\right)\ \text{for some}\ \lambda^{(0)}(0)\neq 0.$$

For $\epsilon\neq 0$ the system \eqref{eq:HC-system} has two regular singular points at $x=0$ and $x=\epsilon$. 
At each one of them, the local information about the system is carried by a formal invariant and a monodromy (holonomy) operator. 
On the other hand, for $\epsilon=0$ the corresponding information about the irregular singularity at $x=0$ is carried by
a formal invariant and by a pair of non-linear operators. 
Our main goal is to explain the relation between these two distinct phenomena, and to show how the Stokes operators are related to the monodromy operators. The principal thesis is, that while the monodromy operators diverge when $\epsilon\to 0$, they
each accumulate to a 1-parameter family of ``wild monodromy operators'' which encode the Stokes phenomenon  (Theorem~\ref{theorem:HC-accumulation}). It is expected that this ``wild monodromy'' should have Galoisian interpretation.

Along the way, we provide a formal normal form and a sectoral normalization theorem for the family (Theorem~\ref{theorem:HC-normalization}), an analytic classification (Theorem~\ref{theorem:HC-classification}), and a decomposition of the monodromy operators (Theorem~\ref{theorem:HC-decomposition}).

In Section~\ref{section:HC-linear}, we illustrate all this on the example of traceless $2\times 2$ linear differential systems
\begin{equation}\label{eq:HC-wm-linear}
 x(x-\epsilon)\tfrac{dy}{dx}=A(x,\epsilon)y,\qquad A(x,\epsilon)\in \mfr{sl}_2(\C),\quad\det A(0,0)\neq 0,
\end{equation}
for which our description follows from the more general work of Lambert and Rousseau \cite{LR, HLR}. 
Here the relation between the monodromy and the Stokes phenomenon can be summarized as:
\begin{theorem*}
When $\epsilon\to 0$ the elements of the monodromy group of the system \eqref{eq:HC-wm-linear} accumulate to generators of the wild monodromy group of the limit system (that is the group generated by the Stokes operators and the exponential torus).
\end{theorem*}
The linear case can be kept in mind as a leading example of which the general non-linear case is a close analogy. 

An important example of a confluent family of systems \eqref{eq:HC-system}, which in fact motivated this study, is the degeneration of the sixth Painlevé equation to the fifth one, 
presented in Section~\ref{section:HC-PV}. A more detailed treatment of this confluence will be the subject of an upcoming article \cite{Kl4}.

\subsection*{Acknowledgments}
This paper was inspired by the works of C.~Rousseau and L.~Teyssier \cite{RT}, C.~Lambert and C.~Rousseau \cite{LR}, and A.~Bittmann \cite{Bit16i,Bit16ii,Bit16iii}.
It was written during my stay at Centre de Recherches Mathématiques at Université de Montréal. 
I want to thank Christiane Rousseau for her support and the CRM for its hospitality.

\goodbreak

\section{The foliation and its formal invariants}

The family of systems \eqref{eq:HC-system} defines a family of singular foliations in the $(y,x)$-space, leaves of which are the solutions.
We associate to \eqref{eq:HC-system} a family of vector fields tangent to the foliations
\begin{equation}\label{eq:HC-vectorfield}
Z_{H,\epsilon}(y,x)=x(x-\epsilon)\partial_x+X_{H,x,\epsilon}(y),
\end{equation}
where
\begin{equation}\label{eq:HC-hamiltonianvectorfield}
X_{H,x,\epsilon}(y)=\tfrac{\partial H}{\partial y_2}(y,x,\epsilon)\partial_{y_1}-\tfrac{\partial H}{\partial y_1}(y,x,\epsilon)\partial_{y_2}. 
\end{equation}
The vector field $Z_{H,0}(y,x)$ has a saddle-node type singularity at $(y,x)=0$, i.e. its linearization matrix has one zero eigenvalue, corresponding to the $x$-direction.
It follows from the Implicit Function Theorem that, for small $\epsilon\neq 0$, $Z_{H,\epsilon}$ has two singular points $(y_0(0,\epsilon),0)$ and $(y_0(\epsilon,\epsilon),\epsilon)$
bifurcating from $(y_0(0,0),0)=0$ and depending analytically on $\epsilon$.
The aim of this paper is a study of their confluence when $\epsilon\to 0$.

\smallskip

The two singularities of $Z_{H,\epsilon}$ have each a strong invariant manifold $\sY_0=\{(y,x):x=0\}$, resp. $\sY_\epsilon=\{(y,x):x=\epsilon\}$. 
Away of these invariant manifolds the vector field $Z_{H,\epsilon}$ is transverse to the 
\emph{fibration}
with fibers $\sY_c=\{(y,x):x=c\}$.
The $(y,x)$-space is endowed with a Poisson structure associated to the 2-form 
\begin{equation}
 \omega=dy_1\wedge dy_2,
\end{equation}
the restriction of which on each fiber $\sY_c$ is symplectic.
The vector field $Z_{H,\epsilon}$ is transversely Hamiltonian with respect to this fibration, the form $\omega$, and the Hamiltonian function $H(y,x,\epsilon)$.

\subsection{Fibered changes of coordinates}

We consider the problem of analytic classification of families of systems \eqref{eq:HC-system}, or orbital analytic classification of vector fields \eqref{eq:HC-vectorfield},
with respect to \emph{fiber-preserving} (shortly \emph{fibered}) \emph{changes of coordinates} 
$$(y,x,\epsilon)=(\Phi(u,x,\epsilon),x,\epsilon).$$
Such a change of coordinate transforms a system \eqref{eq:HC-systemshort}
to a system
\begin{align}
 x(x-\epsilon)\frac{du}{dx}&=(D_u\Phi)^{-1} J\transp{(D_yH)}\circ\Phi-x(x-\epsilon)(D_u\Phi)^{-1}\tfrac{\partial\Phi}{\partial x}\nonumber \\
&=(\det D_u\Phi)^{-1} J\transp{(D_u(H\circ\Phi))}-x(x-\epsilon)(D_u\Phi)^{-1}\tfrac{\partial\Phi}{\partial x},\label{eq:HC-changeofcoordinates}
\end{align}
using the identity 
\begin{equation}\label{eq:HC-PJP}
 PJ\transp{P}=\det P\cdot J\quad\text{for any $2\!\times\!2$ matrix $P$}.
\end{equation}

\begin{definition}
We call a fibered transformation $\Phi$ \emph{transversely symplectic} if $\det(D_u\Phi)\equiv 1$, i.e. if it preserves the restriction of $\omega$ to each fiber $\sY_x$.
\end{definition}

\begin{definition}\label{definition:HC-analequiv}
Two systems \eqref{eq:HC-system} with Hamiltonian functions $\frac{H(y,x,\epsilon)}{x(x-\epsilon)}$ and $\frac{\tilde H(u,x,\epsilon)}{x(x-\epsilon)}$ are called \emph{analytically  equivalent} if there exists an analytic germ of a transversely symplectic transformation $y=\Phi(u,x,\epsilon)$ 
that is analytic in $(u,x,\epsilon)$ and transforms one system to another:  $\Phi^*Z_{H,\epsilon}=Z_{\tilde H,\epsilon}$.
\end{definition}

\begin{lemma}\label{lemma:HC-transhamilton}
If a transformation $y=\Phi(u,x,\epsilon)$ is transversely symplectic, then the transformed system \eqref{eq:HC-changeofcoordinates} is transversely Hamiltonian w.r.t. 
$\omega=du_1\wedge du_2$.
\end{lemma}

\begin{proof}
 It is enough to show that the system
$\frac{du}{dx}=(D_u\Phi)^{-1}\tfrac{\partial\Phi}{\partial x}$ is transversely Hamiltonian, that is, denoting 
$\left(\begin{smallmatrix} f_1 \\[3pt] f_2 \end{smallmatrix}\right):=(D_u\Phi)^{-1}\tfrac{\partial\Phi}{\partial x}$,
to show that $\partial{u_1}f_1+\partial_{u_2}f_2=0$.
Using the identity \eqref{eq:HC-PJP}, we can express
$ f_1=\tfrac{1}{\det D_u\Phi}\transp(\partial_{u_2}\!\Phi)J\partial_x\Phi,\ \
 f_2=-\tfrac{1}{\det D_u\Phi}\transp(\partial_{u_1}\!\Phi)J\partial_x\Phi,$
hence
\begin{align*}
\partial_{u_1}f_1+\partial_{u_2}f_2&=
\tfrac{1}{\det D_u\Phi}\big[ \transp(\partial_{u_2}\!\Phi)J\partial_x\partial_{u_1}\!\Phi -\transp(\partial_{u_1}\!\Phi)J\partial_x\partial_{u_2}\!\Phi\big]\\
&=\tfrac{1}{\det D_u\Phi}\partial_x\big[\transp(\partial_{u_2}\!\Phi)J\partial_{u_1}\!\Phi\big]=\tfrac{1}{\det D_u\Phi}\partial_x\det D_u\Phi=0.
\end{align*}
\end{proof}

\subsection{The formal invariant $\chi(h,x,\epsilon)$}\label{paragraph:HC-chi}

\begin{theorem}[Siegel]\label{theorem:HC-siegel}
Let  $H:(\C^2,0)\to(\C,0)$ have a non-degenerate critical point at $0$, and let $\omega$ be a symplectic volume form. 
There exists an analytic system of coordinates $u=(u_1,u_2)$ in which 
$$\omega=du_1\wedge du_2,\qquad\text{and}\qquad H=G_H(u_1u_2).$$
The function $G_H$ is uniquely determined by the pair $(H,\omega)$
up to the involution
\begin{equation}\label{eq:HC-involution}
 G_H(u_1u_2)\mapsto G_H(-u_1u_2),
\end{equation}
induced by the symplectic change of variable $J:(u_1,u_2)\mapsto(u_2,-u_1)$.
The pair $(G_H,\omega)$ is called the \emph{Birkhoff-Siegel normal form} of the pair $(H,\omega)$.
Moreover, if $(H,\omega)$ depend analytically on a parameter, then so does $G_H$ and the change of coordinates.
\end{theorem}

\begin{proof}
While not explicitly stated, the existence part of the theorem is originally proved by Siegel in \cite[chap. 16 and 17]{SM}. See also \cite{Vey77} and \cite{FS94}.
The uniqueness can be seen by expressing $G_H(h)$ in terms of a period map over a vanishing cycle, see Section~\ref{section:HC-periodmap} below.
\end{proof}

\begin{remark}
\begin{itemize}
\item[--] The Theorem~\ref{theorem:HC-siegel} 
provides the existence of an analytic transformation of a Hamiltonian vector field $\dot y=J\transp(D_yH)$ in dimension 2 to its Birkhoff normal form $\dot u=J\transp(D_uG_H)$, 
with 
$$G_H(h)=\lambda h+\ldots,$$
where $\pm\lambda\neq 0$ are the eigenvalues of the linear part $JD_y^2H(0)$. The involution \eqref{eq:HC-involution}
corresponds to the freedom of choice of the eigenvalue $\lambda$.  
\item[--] The change of coordinates is far from unique. Indeed, the flow of any vector field 
$\xi=a(u_1u_2)\big(u_1\partial_{u_1}-u_2\partial_{u_2}\big)$ preserves the normal form.
\end{itemize}
\end{remark}

Let $H=H(y,x,\epsilon)$ be our germ.
By the implicit function theorem, for each small $(x,\epsilon)$, the function  $H(\cdot,x,\epsilon)$ has an isolated non-degenerate critical point $y_0(x,\epsilon)$, depending analytically on $(x,\epsilon)$.  
Let $y=\Phi(u,x,\epsilon)$ be the transformation to the Birkhoff-Siegel normal form for the function $y\mapsto H(y,x,\epsilon)$ and the form $\omega=dy_1\wedge dy_2$, depending analytically on $(x,\epsilon)$, i.e.
$$H(\cdot,x,\epsilon)\circ \Phi(u,x,\epsilon)=G_H(u_1u_2,x,\epsilon),\qquad \det D_u\Phi(u,x,\epsilon)\equiv 1.$$
By \eqref{eq:HC-changeofcoordinates}, it brings the system \eqref{eq:HC-system} to a \emph{prenormal form}
\begin{align}
 x(x-\epsilon)\frac{du}{dx}&=J\transp{(D_uG_H)}+O(x(x-\epsilon))\nonumber \\
    &=\chi(u_1u_2,x,\epsilon)\left(\begin{smallmatrix} 1 & 0 \\[3pt] 0 & -1 \end{smallmatrix}\right) u+O(x(x-\epsilon)),
\label{eq:HC-prenormalform}
\end{align}
where
\begin{equation}\label{eq:HC-chi}
 \chi(h,x,\epsilon)=\chi^{(0)}(h,\epsilon)+x\chi^{(1)}(h,\epsilon):= \tfrac{\partial G_H}{\partial h}(h,x,\epsilon)\mod x(x-\epsilon),
\end{equation}
$h=u_1u_2$, or equivalently, 
\begin{equation}\label{eq:HC-chi2}
\chi(h,x,\epsilon)=
\left\{\!\!\begin{array}{l}   
\tfrac{1}{\epsilon}\big[x \tfrac{\partial G_{H}}{\partial h}(h,\epsilon,\epsilon)-(x-\epsilon)\tfrac{\partial G_{H}}{\partial h}(h,0,\epsilon)\big],\quad\epsilon\neq 0,\\[6pt]
\tfrac{\partial G_{H}}{\partial h}(h,0,0)+x\tfrac{\partial^2G_{H}}{\partial h\partial x}(h,0,0),\quad\epsilon=0.
\end{array}\right.
\end{equation}

\begin{definition}
The function $\chi(h,x,\epsilon)$ is called a \emph{formal invariant} of the system \eqref{eq:HC-system}.
\end{definition}

For $\epsilon\neq 0$ the formal invariant $\chi$ is completely determined by
the functions $G_{H}(\cdot,0,\epsilon)$ and $G_{H}(\cdot,\epsilon,\epsilon)$ ,
which are \emph{analytic invariants of the autonomous Hamiltonian systems $X_{H,0,\epsilon}$, $X_{H,\epsilon,\epsilon}$ \eqref{eq:HC-hamiltonianvectorfield} on the strong invariant manifolds} $\sY_0,\sY_\epsilon$. 

\begin{corollary}\label{proposition:HC-chi}
The formal invariant $\chi(h,x,\epsilon)$ is well-defined
up to the involution
\begin{equation}\label{eq:HC-reflection}
J^*: \chi(h,x,\epsilon)\mapsto -\chi(-h,x,\epsilon),
\end{equation}
induced by the symplectic transformation $u\mapsto Ju$.
It is uniquely determined by the polar part of the Hamiltonian $\frac{H(y,x,\epsilon)}{x(x-\epsilon)}$,
and it is invariant with respect to fibered transversely symplectic changes of coordinates.
\end{corollary}

Let
$$\lambda(x,\epsilon)=\chi(0,x,\epsilon).$$
Then  $\pm\lambda(x,\epsilon)$ are the eigenvalues of the matrix $A(x,\epsilon)=J\transp{D_y^2H(0,x,\epsilon)}$ modulo $x(x-\epsilon)$, see Example~\ref{example:HC-2},
and the involution \eqref{eq:HC-reflection} corresponds to the freedom of choice of the eigenvalue $\lambda$.

\begin{example}[Traceless linear systems]\label{example:HC-2}
A traceless linear system
\begin{equation}\label{eq:HC-linsyst}
 x(x-\epsilon)\frac{dy}{dx}=A(x,\epsilon)y,
\end{equation}
with $\tr A(x,\epsilon)=0$ and $A(0,0)\sim \left(\begin{smallmatrix} \lambda^{(0)}(0) & 0 \\[3pt] 0 & -\lambda^{(0)}(0) \end{smallmatrix}\right)$ for some $\lambda^{(0)}(0)\neq 0$,
is of the form \eqref{eq:HC-systemshort} for the quadratic form $H(y,x,\epsilon)=\tfrac{1}{2} \transp{y}J\transp{A(x,\epsilon)}y$.
Let $\pm\tilde\lambda(x,\epsilon)$ be the eigenvalues of $A(x,\epsilon)$, and let $C(x,\epsilon)$ be a corresponding matrix of eigenvectors of $A(x,\epsilon)$, depending analytically on $(x,\epsilon)$ and normalized so that $\det C(x,\epsilon)=1$.
The change of variable $y=C(x,\epsilon)u$, brings the system \eqref{eq:HC-linsyst} to
\begin{equation*}
 x(x-\epsilon)\frac{du}{dx}=\tilde\lambda(x,\epsilon)\left(\begin{smallmatrix} 1 & 0 \\ 0 & -1 \end{smallmatrix}\right)u+ x(x-\epsilon)C^{-1}\frac{dC}{dx}.
\end{equation*}
Denoting $\lambda(x,\epsilon):=\big(\lambda^{(0)}(\epsilon)+x\lambda^{(1)}(\epsilon)\big)=\tilde\lambda(x,\epsilon)\mod x(x-\epsilon)$,
then we have $\chi(h,x,\epsilon)=\pm\lambda(x,\epsilon)$.  
\end{example}

\subsubsection{Geometric interpretation of the invariant $\chi$.}\label{section:HC-periodmap}

For each small $(x,\epsilon)$, the function  $H(\cdot,x,\epsilon)$ has an isolated non-degenerate critical point $y_0(x,\epsilon)$, depending analytically on $(x,\epsilon)$, with a critical value $h_0(x,\epsilon)$.  
For $(x,\epsilon)$ fixed, $h\in(\C,h_0)$, consider the germ of the level set $S_h(x,\epsilon)=\{y\in(\C^2,y_0(x,\epsilon)): H(y,x,\epsilon)=h\}\subset \sY_x$. 
As a basic fact of the Picard--Lefschetz theory \cite{AVG}, we know that if $h$ is a non-critical value for $H(\cdot,x,\epsilon)$, i.e. $h\neq h_0$, then $S_h(x,\epsilon)$ has the homotopy type of a circle. 
Let $\gamma_h(x,\epsilon)$ depending continuously on $(x,\epsilon)$ be a loop generating the first homology group of $S_h(x,\epsilon)$, the so called \emph{vanishing cycle}.
And let $\mu$ be a 1-form such that 
$\omega=dH\wedge\mu;$ 
its restriction to a non-critical level $S_h(x,\epsilon)$ is called the \emph{Gelfand-Leray} form of $\omega$ and is denoted 
$$\mu=\frac{\omega}{dH}.$$
Its period function over the vanishing cycle
\begin{equation}\label{eq:HC-periodmap}
p(h,x,\epsilon):=\tfrac{1}{2\pi i}\int_{\gamma_h(x,\epsilon)} \frac{\omega}{dH},
\end{equation}
is well-defined up to a sign change (orientation of $\gamma_h$),
and depends analytically on $(x,\epsilon)$ \cite[chap. 10]{AVG}.
Let $G_{H}(\cdot,x,\epsilon)$ be the inverse of the function $h\mapsto \int_{h_0(x,\epsilon)}^h p(s,x,\epsilon)\,ds$. 
Then $(G_H,\omega)$ is the Birkhoff-Siegel normal form of $(H,\omega)$. 

Indeed, the above formula for $G_H$ is invariant with respect to analytic transversely symplectic changes of coordinates:
Supposing that $H=g(y_1y_2,x,\epsilon)$ is in its Birkhoff-Siegel normal form, then the level sets are written as 
$S_h=\{y_1\neq 0,\ y_2=\tfrac{g^{\circ(-1)}(h,x,\epsilon)}{y_1}\}$, and 
$\frac{\omega}{dH}=\frac{dy_1}{y_1\cdot \frac{\partial g}{\partial(y_1y_2)}\circ g^{\circ(-1)}(h,x,\epsilon)}$, and therefore
$p(h)=\big(\frac{\partial g}{\partial(y_1y_2)}\big)^{-1}\circ g^{\circ(-1)}(h,x,\epsilon)$, i.e. $G_H=g$. 


\medskip
The above formula for the the Birkhoff-Siegel normal form and hence for the formal invariant $\chi$ involves a double inversion which makes it difficult to calculate. 
The following proposition, which will be proved in Section~\ref{section:HC-prooflemma}, allows to determine it in some special cases. 
This will be useful in the case of the fifth Painlevé equation (Section~\ref{section:HC-PV}).

\begin{proposition}[Birkhoff-Siegel normal form of an autonomous Hamiltonian system]\label{proposition:HC-lemma}~

\noindent
Let $H(y)$ be of the form
$$H(y)=G(h)+y_i\Delta(y_i,h),\quad h=y_1y_2,$$
for some $i\in\{1,2\}$, with $G,\Delta$ analytic germs, and
$$G(h)=\lambda h+O(h^2),\quad \lambda\neq 0.$$
Then $(G,\omega)$ is the Birkhoff-Siegel normal form for the pair $(H,\omega)$, $\omega=dy_1\wedge dy_2$.
\end{proposition}

\begin{corollary}[Invariant $\chi$ for $\epsilon=0$]
For $\epsilon=0$, suppose that
$$H(y,x,0)=G^{(0)}(h,0)+xH^{(1)}(y,0)+O(x^2),\quad h=y_1y_2,$$
where $G^{(0)}(h,0)=\lambda^{(0)}(0) h+O(h^2)$, $\lambda^{(0)}(0)\neq 0$.
Write 
$$H^{(1)}(y,0)=G^{(1)}(h,0)+y_1\Delta_1(y_1,h,0)+y_2\Delta_2(y_2,h,0),$$
with $\Delta_i(y_i,h,0)=O(y_i)$.
Then $$\chi(h,x,0)=\tfrac{\partial}{\partial h}\big(G^{(0)}(h,0)+x\,G^{(1)}(h,0)\big)$$
is the formal invariant of the vector field $Z_{H,0}=x^2\partial_x+X_{H}$ associated to $H$.
\end{corollary}

\begin{proof}
Consider a deformation 
$$H(y,x,\epsilon)=G^{(0)}(h,0)+xG^{(1)}(h,0)+xy_1\Delta_1(y_1,h,0)+(x-\epsilon)y_2\Delta_2(y_2,h,0)+O(x(x-\epsilon)),$$
and calculate the Birkhoff-Siegel invariants for $H(y,0,\epsilon)$, $H(y,\epsilon,\epsilon)$, using Proposition~\ref{proposition:HC-lemma}.
\end{proof}

\subsection{Model system (formal normal form)}

\begin{definition}[Model family]\label{definition:HC-normalform}
Let $\chi(h,x,\epsilon)$ be the formal invariant of the system \eqref{eq:HC-system}.
The \emph{model family} (\emph{formal normal form}) for the the system \eqref{eq:HC-system} is the family of systems
\begin{equation}\label{eq:HC-normalform}
\begin{aligned}
x(x-\epsilon)\tfrac{du_1}{dx}&=\ \chi(u_1u_2,x,\epsilon)\cdot u_1\\
x(x-\epsilon)\tfrac{du_2}{dx}&=-\chi(u_1u_2,x,\epsilon)\cdot u_2,
\end{aligned}
\end{equation}
which is Hamiltonian with respect to the Hamiltonian function $\frac{G(u_1u_2,x,\epsilon)}{x(x-\epsilon)}$,
\begin{equation}
 G(h,x,\epsilon)=\int_0^h \chi(s,x,\epsilon)\,ds.
\end{equation}
The formal normal form of the family $Z_{H,\epsilon}$ is the associated family of vector fields
\begin{equation}\label{eq:HC-vfnormalform}
Z_{G,\epsilon}=x(x-\epsilon)\partial_x+\chi(u_1u_2,x,\epsilon)\big(u_1\partial_{u_1}-u_2\partial_{u_2}\big).
\end{equation}

The system \eqref{eq:HC-normalform} is \emph{integrable} with the function $h(u)=u_1u_2$ being its first integral, $Z_{G,\epsilon}\cdot h=0$.
The general solutions of \eqref{eq:HC-normalform} are of the form
\begin{equation}\label{eq:HC-u}
\begin{aligned}
u_1(x,\epsilon;c)&=c_1 E_\chi(c_1c_2,x,\epsilon),\\
u_2(x,\epsilon;c)&=c_2 E_\chi(c_1c_2,x,\epsilon)^{-1},
\end{aligned}
\qquad 
c=(c_1,c_2)\in\C^2,
\end{equation}
where
\begin{equation}\label{eq:HC-E}
E_\chi(h,x,\epsilon)= 
\left\{\!\!\begin{array}{ll} 
x^{-\frac{\chi^{(0)}(h,\epsilon)}{\epsilon}}(x-\epsilon)^{\frac{\chi^{(0)}(h,\epsilon)}{\epsilon}+\chi^{(1)}(h,\epsilon)}, & \text{for $\epsilon\neq 0$},\\[6pt]
e^{-\frac{\chi^{(0)}(h,0)}{x}}x^{\chi^{(1)}(h,0)}, & \text{for $\epsilon= 0$}.
\end{array}\right.
\end{equation}
\end{definition}

\section{Formal and sectoral normalization theorem}

Throughout the text we will denote
\begin{equation}
 \sY=\{|y|<\delta_y\},\qquad \sU=\{|u|<\delta_u\},\qquad \sX=\{|x|<\delta_x\},\qquad \sE=\{|\epsilon|<\delta_\epsilon\},
\end{equation}
for some $\delta_y, \delta_u, \delta_x,\delta_\epsilon>0$, and implicitly suppose that $\delta_\epsilon<<\delta_x$ so that the singular points $x=0,\epsilon$ are both well inside $\sX$.

\begin{definition}[Family of spiraling sectoral domains $\sX_\pm(\epsilon)$, $\epsilon\in \sE_\pm$]\label{definition:HC-X}
Let $\eta>0$ be  an arbitrarily small constant, and let $\delta_x>>\delta_\epsilon>0$ be radii of small discs at 0 in the $x$-and $\epsilon$-space.
Let $\lambda(x,\epsilon)=\chi(0,x,\epsilon),$ and let
\begin{equation}
\sE_\pm:=\{ |\epsilon|<\delta_\epsilon,\ |\arg\big(\pm\tfrac{\epsilon}{\lambda(0,0)}\big)|<\pi-2\eta \},  
\end{equation}
be two sectors in the $\epsilon$-space.
For $\epsilon\in \sE_\pm$ define a domain 
$$\sX_\pm(\epsilon)$$
in the $x$-space as a simply connected ramified domain spanned by the complete real trajectories of the vector fields
\begin{equation}\label{eq:HC-xvectfield}
 e^{i\omega_\pm}\cdot\tfrac{x(x-\epsilon)}{\lambda(x,\epsilon)}\partial_x
\end{equation}
that never leave the disc of radius $\delta_x$,
where the phase $\omega_\pm$ varies continuously in the interval
\begin{equation}\label{eq:HC-omega}
\left\{\!\!\begin{array}{l}   
\max\{0,\arg(\frac{\pm\epsilon}{\lambda(0,0)})\}-\frac{\pi}{2}+\eta<\omega_\pm<\min\{0,\arg(\frac{\pm\epsilon}{\lambda(0,0)})\}+\frac{\pi}{2}-\eta,\quad\text{for $\epsilon\neq 0$},\\[6pt]
|\omega_\pm|<\frac{\pi}{2}-\eta,\quad\text{for $\epsilon=0$.}
\end{array}\right.
\end{equation}
The constraints \eqref{eq:HC-omega} on the variation of $\omega_\pm$ are such that the real dynamics of the vector field \eqref{eq:HC-xvectfield} and the asymptotic behavior of the solutions
\eqref{eq:HC-u} would not change drastically depending on $\omega_\pm$. Namely, for $\epsilon\neq 0$:
\begin{itemize}
 \item The point $x=\epsilon$ is repulsive when $|\arg \frac{\epsilon}{\lambda(x,\epsilon)} -\omega_\pm|<\frac{\pi}{2}$ and attractive when
 $|\arg \frac{-\epsilon}{\lambda(x,\epsilon)} -\omega_\pm|<\frac{\pi}{2}$, and vice-versa for the point $x=0$.
 \item The $u_1$-component of the solution \eqref{eq:HC-u} tends  to $0$ along a negative real trajectory of \eqref{eq:HC-xvectfield}  and to $\infty$ along a positive real trajectory for $|\omega_\pm|<\tfrac{\pi}{2}$, and vice-versa for the $u_2$-component. 
\end{itemize}
For $\epsilon=0$, the domain $\sX_+(0)=\sX_-(0)$ consists of a pair of overlapping sectoral domains $\sX^{\uppie}(0), \sX^{\downpie}(0)$ of opening $2\pi-2\eta$ with a common point at $x=0$. See Figure~\ref{figure:HC-X}.
\end{definition}

Denoting $x_{1,\pm}(\epsilon)$ the \emph{attractive} equilibrium point of \eqref{eq:HC-xvectfield} and $x_{2,\pm}(\epsilon)$ the \emph{repulsive} one,
\begin{equation}\label{eq:HC-x12}
x_{1,+}(\epsilon)=x_{2,-}(\epsilon)=0, \qquad x_{1,-}(\epsilon)=x_{2,+}(\epsilon)=\epsilon, 
\end{equation}
then 
$$u_i(x,\epsilon)\to \infty,\ u_{j}(x,\epsilon)\to 0,\ (j=3-i),\ \text{ when }\ x\to x_{i,\pm}(\epsilon)\ \text{along a real trajectory of \eqref{eq:HC-xvectfield}}.$$

\begin{figure}[t]
\centering
\includegraphics[width=.9\textwidth]{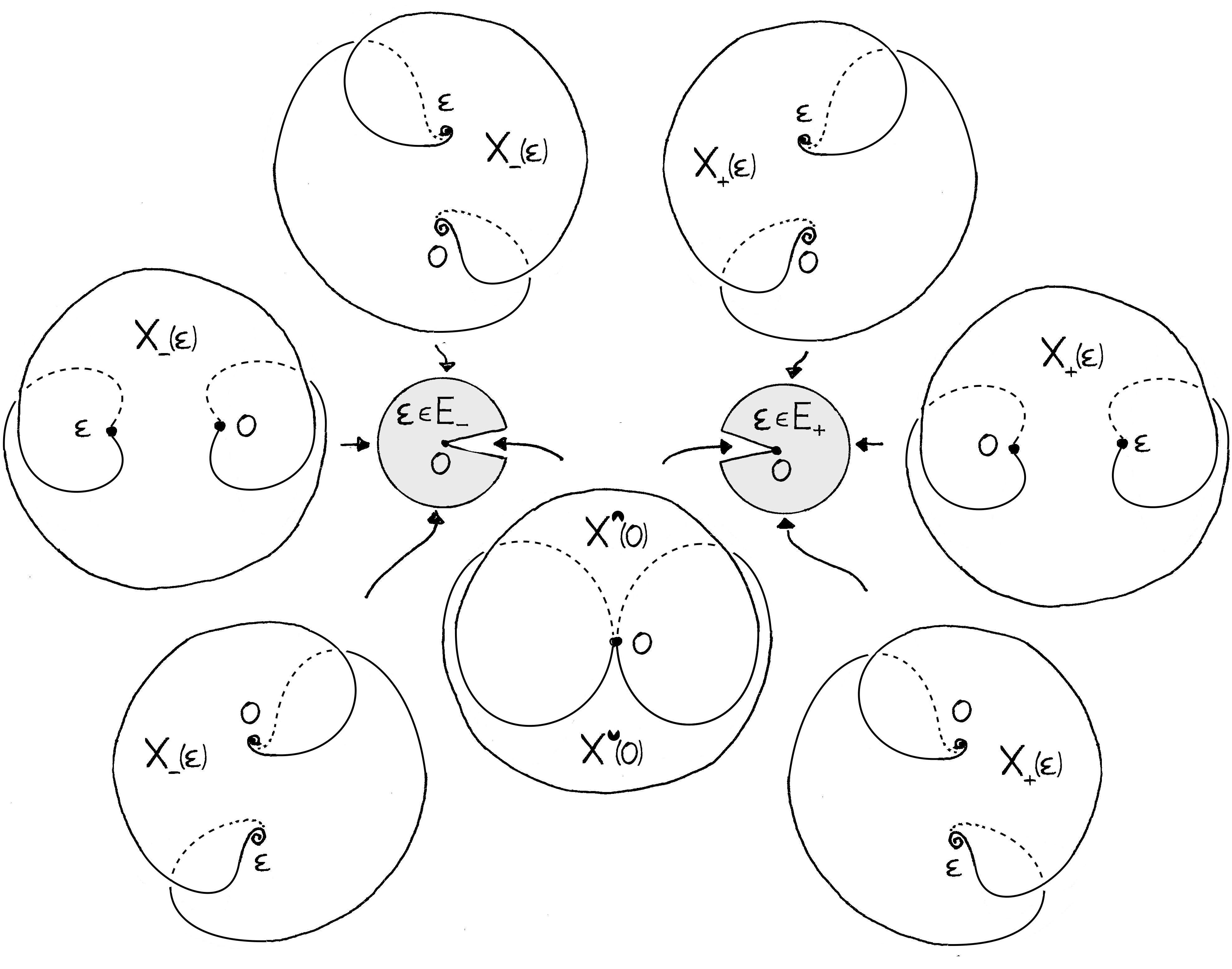}
\caption{The domains $\sX_\pm(\epsilon)$ in dependence on $\epsilon\in\sE_\pm$. (Picture with $\lambda^{(0)}(0)=1$.)}
\label{figure:HC-X}
\end{figure}

Before giving a general theorem on sectoral normalization for the parametric family \eqref{eq:HC-system}, let us first state it for the limit system with $\epsilon=0$
which has an irregular singularity of Poincaré rank 1 at $x=0$.

\begin{theorem}[Formal and sectoral normalization at $\epsilon=0$]\label{theorem:HC-normalization0}
\noindent
The system \eqref{eq:HC-system} with $\epsilon=0$ can be brought to its formal normal form \eqref{eq:HC-normalform} 
through a formal transversely symplectic change of coordinates 
\begin{equation}\label{eq:HC-hatPhi0}
(y,x)=(\hat{\mbf\Psi}(u,x,0),x),\qquad \hat{\mbf\Psi}(u,x,0)=\sum_{k\geq 0}\psi^{(k)}(u)x^k,
\end{equation}
where $\psi^{(k)}(u)$ are analytic in $u$ on a fixed neighborhood 
$\sU$ of $0$.
This formal series is generally divergent, but it is Borel 1-summable, with a pair of Borel sums $\mbf\Psi^{\uppie}(u,x,0)$ and $\mbf\Psi^{\downpie}(u,x,0)$
defined respectively above the sectors $x\in \sX^{\uppie}(0), \sX^{\downpie}(0)$
of Definition~\ref{definition:HC-X} (for some $0<\eta<\tfrac{\pi}{2}$ arbitrarily small and some $\delta_x>0$ depending on $\eta$), and $u\in \sU$.
The fibered sectoral transformations 
$(y,x)=(\mbf\Psi^\bullet(u,x,0)$, $\bullet=\uppie,\downpie$,
are transversely symplectic
and
bring the system \eqref{eq:HC-system} with $\epsilon=0$ to its formal normal form.
\end{theorem}

The Theorem~\ref{theorem:HC-normalization0} is originally due to Takano \cite{Ta90} for systems \eqref{eq:HC-system} whose formal invariant is of the form $\chi(h,x)=\lambda^{(0)}+x\chi^{(1)}(h)$.
In the case of the irregular singularity of the fifth Painlevé equation it was proved earlier by Takano \cite{Ta}.   
Some similar and closely related theorems are due to Shimomura \cite{Shi}, Yoshida \cite{Yo1}, and recently by Bittmann \cite{Bit16i,Bit16ii}, which apply to doubly resonant systems 
$x^2\frac{dy}{dx}=F(y,x)$, with $D_yF(0,0)=\lambda^{(0)}(0)\left(\begin{smallmatrix} 1&\\&-1\end{smallmatrix}\right)$
under a condition on positivity of $\tr\frac{\partial}{\partial x} D_yF(0,0)$. This condition is not satisfied for Hamiltonian systems \eqref{eq:HC-system} 
but nevertheless allows to treat Painlevé equations.

\begin{theorem}[Formal and sectoral normalization]\label{theorem:HC-normalization}
Let $Z_{H,\epsilon}(y,x)$ be a family of vector fields \eqref{eq:HC-vectorfield} and let $\chi(h,x,\epsilon)$ be their formal invariant.

\smallskip
\noindent
\textbf{(i)}
There exists a formal transversely sympectic change of coordinates 
$(y,x,\epsilon)=(\hat{\mbf\Psi}(u,x,\epsilon),x,\epsilon)$ 
written as a formal power series
\begin{equation}\label{eq:HC-hatPhi}
 \hat{\mbf\Psi}(u,x,\epsilon)=\psi^{(0)}(u,\epsilon)+x\psi^{(1)}(u,\epsilon)+x(x-\epsilon)\sum_{k,l\geq 0}\tilde\psi^{(kl)}(u)x^k\epsilon^l,
\end{equation}
with $\psi^{(0)}(u,\epsilon)$, $\psi^{(1)}(u,\epsilon)$ analytic in $(u,\epsilon)$, and $\tilde\psi^{(kl)}(u)$ analytic in $u$ on a fixed neighborhood 
$\sU=\{|u_1|,|u_2|<\delta_u\}$ of $0$,
which brings $Z_\epsilon$ to its formal normal form \eqref{eq:HC-vfnormalform}.
The coefficients $\tilde\psi^{(kl)}$ grow at most factorially in $k+l$:
$$\max_{u\in \sU}\|\tilde\psi^{(kl)}(u)\|\leq L^{k+l}(k+l)!\quad\text{for some}\quad L>0.$$

\smallskip
\noindent
\textbf{(ii)}
There exists a transversely symplectic fibered change of coordinates $(y,x,\epsilon)=(\mbf\Psi_\pm(u,x,\epsilon),x,\epsilon)$,
with $\mbf\Psi_\pm(u,0,\epsilon)=\psi^{(0)}(u,\epsilon)$ \eqref{eq:HC-hatPhi},
defined for $x$ in the spiraling domain $\sX_\pm(\epsilon)$, $\epsilon\in\sE_\pm$, of Definition~\ref{definition:HC-X} (for some $0<\eta<\tfrac{\pi}{2}$ arbitrarily small and some $\delta_x,\delta_\epsilon>0$ depending on $\eta$), and for $u\in \sU$,
which brings $Z_\epsilon$ to its formal normal form \eqref{eq:HC-vfnormalform}.
It is uniformly continuous on 
$$\sXE_\pm=\{(x,\epsilon)\mid x\in \sX_\pm(\epsilon)\}$$
and analytic on its interior.
When $\epsilon$ tends radially to $0$ with $\arg\epsilon=\beta$, then 
$\mbf\Psi_\pm(u,x,\epsilon)$ converges to $\mbf\Psi_\pm(u,x,0)$ uniformly on compact sets of the sub-domains 
$\lim_{\substack{\epsilon\to 0\\\arg\epsilon=\beta}}\sX_{\pm}(\epsilon)\subseteq \sX(0)$.
Note that in our notation $\mbf\Psi_\pm(u,x,0)$ consists of a pair of sectoral transformations $\mbf\Psi^{\uppie}(u,x,0)$ and $\mbf\Psi^{\downpie}(u,x,0)$; it is a functional cochain using the terminology of \cite{IlYa}.

\smallskip
\noindent
\textbf{(iii)}
Let $\Tilde\Psi(u,x,\epsilon)$ be an analytic extension of the function given by the convergent series
$$\tilde\Psi(u,x,\epsilon)=\sum_{k,l\geq 0} \frac{\tilde\psi^{(kl)}(u)}{(k+l)!}x^k\epsilon^l, \quad\text{$\tilde\psi^{(kl)}$ as in \eqref{eq:HC-hatPhi}}.$$ 
For each point $(x,\epsilon)$, for which there is  $\theta \in\,]\!-\!\frac{\pi}{2},\frac{\pi}{2}[$ such that
$\mbox{$\mathbf{S}_\theta\cdot(x,\epsilon)$}\subseteq \sXE_\pm$, with $\mathbf{S}_\theta\subset\C$ denoting the circle  through the points $0$ and $1$ with  center on $e^{i\theta}\R^+$,
we can express $\mbf\Psi_\pm(u,x,\epsilon)$ through the 
following Laplace transform of $\tilde\Psi$:
\begin{equation}\label{eq:HC-Phi}
\mbf\Psi_\pm(u,x,\epsilon)=\psi^{(0)}(u,\epsilon)+x\psi^{(1)}(u,\epsilon)+x(x-\epsilon)\int_0^{+\infty e^{i\theta}}\!\!\!\Tilde\Psi(u,sx,s\epsilon)\,e^{-s}\,ds.
\end{equation}
In particular, $\mbf\Psi_\pm(u,x,0)$ is the pair of sectoral Borel sums $\mbf\Psi^{\uppie}(u,x,0)$, $\mbf\Psi^{\downpie}(u,x,0)$ of the formal series $\hat{\mbf\Psi}(u,x,0)$.

As a consequence, $\mbf\Psi_\pm$ and $\hat{\mbf\Psi}$ satisfy the same $(\partial_u,\partial_x,\partial_\epsilon)$-differential relations with meromorphic coefficients.
\end{theorem}

The proof will be given in Section \ref{sec:HC-proof}.

The transformations $\mbf\Psi_\pm$ and $\hat{\mbf\Psi}$ are unique up to left composition with an analytic symmetry of the model system, see Corollary~\ref{corollary:HC-21}.

\begin{corollary}
The system \eqref{eq:HC-system} possesses:
\begin{itemize}
 \item[(i)] a formal first integral given by $h\circ\hat{\mbf\Psi}^{\circ(-1)}(y,x,\epsilon)$, where $\hat{\mbf\Psi}$ as above and $h(u)=u_1u_2$
is a first integral of the model system,
\item[(ii)] an actual first integral given by $h\circ\mbf\Psi_\pm^{\circ(-1)}(y,x,\epsilon)$ that is bounded and analytic on the domain $\sXE_\pm$. 
\end{itemize}
\end{corollary}

\begin{definition}
 The solution $y=\mbf\Psi_\pm(0,x,\epsilon)$ is called \emph{ramified center manifold}. It is the unique solution that is bounded on $\sX_\pm(\epsilon)$ (cf. \cite{Kl2}).
\end{definition}

\begin{remark}
In the variable $x=\epsilon z$, the system \eqref{eq:HC-systemshort} takes the form of a singularly perturbed system
\begin{equation*}
\epsilon z(z-1)\tfrac{dy}{dz}=J\transp{(D_yH)(y,\epsilon z,\epsilon)}.
\end{equation*}
The domains $z\in\frac{1}{\epsilon}\sX_\pm(\epsilon)$, $\epsilon\in\sE_\pm$, then correspond to the Stokes domains in the sense of exact WKB analysis \cite{KT},
where the Stokes curves would be the real separatrices of the point $z=\infty$ of the vector field \eqref{eq:HC-xvectfield} $e^{i(\omega_\pm+\arg\epsilon)}\frac{z(z-1)}{\lambda(0,0)}\partial_z$ with a fixed phase $\omega_\pm$ \eqref{eq:HC-omega}.
\end{remark}

\goodbreak

\section{Stokes operators and accumulation of monodromy}

We will define several operators acting as transversely symplectic fibered isotropies on the three following foliations given by three different vector fields: 
\begin{itemize}
\item Foliation in the $(u,x)$-space given by the model vector field $Z_{G,\epsilon}$ \eqref{eq:HC-vfnormalform}.
\item Foliation in the $(c,x)$-space, $c$ being the constant of initial condition in \eqref{eq:HC-u},
given by the rectified vector field $Z_{0,\epsilon}=x(x-\epsilon)\partial_x$. 
Note that a fibered isotropy of $Z_{0,\epsilon}$ is necessarily \emph{independent} of $x$; it acts on the $c$-space of \emph{initial conditions} only.
\item Foliation in the $(y,x)$-space given by the original vector field $Z_{H,\epsilon}$ \eqref{eq:HC-vectorfield}.
\end{itemize}

\subsection{Symmetries of the model system: exponential torus}

A \emph{vertical infinitesimal symplectic symmetry} (shortly \emph{infinitesimal symmetry}) of the normal form vector field $Z_{G,\epsilon}$ \eqref{eq:HC-vfnormalform}
is a germ of vector field $\xi$ in the $(u,x)$-space that preserves:
\begin{itemize}
\item[(i)] the $x$-coordinate:\quad $\Cal L_{\xi}x=\xi\cdot x=0$,
\item[(ii)] the symplectic form $\omega=du_1\wedge du_2$: \quad $\Cal L_{\xi}\omega=0$,
\item[(iii)] the vector field $Z_{G,\epsilon}$: \quad $\Cal L_{\xi}Z_{G,\epsilon}=[\xi,Z_{G,\epsilon}]=0.$
\end{itemize}

\begin{lemma}
 A vector field $\xi$ is an infinitesimal symmetry of $Z_{G,\epsilon}$ if and only if $\xi=X_{f,x,\epsilon}$ is a Hamiltonian vector field with respect to $\omega$ for a first integral
$f(u,x,\epsilon)$ of $Z_{G,\epsilon}$: 
$$Z_{G,\epsilon}\cdot f=0.$$ 
\end{lemma}

\begin{proof}
The conditions \textit{(i)} and \textit{(ii)} say that $\xi=a_1(u,x,\epsilon)\partial_{u_1}+a_2(u,x,\epsilon)\partial_{u_2}$ with 
$\frac{\partial a_1}{\partial u_1}+\frac{\partial a_2}{\partial u_2}=0$,
i.e. $a_1=\frac{\partial f}{\partial u_2}$, $a_2=-\frac{\partial f}{\partial u_1}$ for some $f$, and $\xi=\frac{\partial f}{\partial u_2}\partial_{u_1}-\frac{\partial f}{\partial u_1}\partial_{u_2}=X_f$.

The condition \textit{(iii)} says that
$$0=Z_{G,\epsilon}\cdot D_uf-X_{f,x,\epsilon}\cdot D_uG=D_u(Z_{G,\epsilon}\cdot f).$$
Up to a translation $f(u,x,\epsilon)\mapsto f(u,x,\epsilon)-f(0,x,\epsilon)$ which does not affect $X_{f,x,\epsilon}$, this condition is equivalent to
$Z_{G,\epsilon}\cdot f=0$.
\end{proof}

The vector field $Z_{G,\epsilon}$ has the following obvious first integrals (cf. \eqref{eq:HC-u}):
\begin{equation}\label{eq:HC-c}
 c_1(u,x,\epsilon)=u_1\cdot E_\chi(h,x,\epsilon)^{-1},\qquad c_2(u,x,\epsilon)=u_2\cdot E_\chi(h,x,\epsilon),
\end{equation}
and
\begin{equation*}
 h(u)=u_1u_2=c_1c_2,
\end{equation*}
where $E_\chi(h,x,\epsilon)$ is as in \eqref{eq:HC-E}.
Clearly, any function of $c=(c_1,c_2)$
is again a first integral, and since $c$ defines local coordinates on the space of leaves (space of initial conditions), the converse is also true. 
Note that the map $c:u\mapsto c(u,x,\epsilon)$ conjugates the vector field $Z_{G,\epsilon}$ to the ``rectified'' vector field $Z_{0,\epsilon}=x(x-\epsilon)\partial_x$ in the $(c,x)$-space:
\begin{equation*}
 Z_{G,\epsilon}=c^*(x(x-\epsilon)\partial_x).
\end{equation*}
It turns out that analytic first integrals are functions of $h=c_1c_2$ only.

\begin{proposition}\label{proposition:HC-firstintegrals}
If $f(u,x,\epsilon)$ is an analytic (resp. meromorphic) first integral of $Z_{G,\epsilon}$ on some neighborhood $\sU\times\sX\times\sE$ of 0, then $f=F(u_1u_2,\epsilon)$ with $F$ analytic (resp. meromorphic).   
\end{proposition}

\begin{proof}[Proof of Proposition~\ref{proposition:HC-firstintegrals}]
On one hand, $c_1(u,x,\epsilon)$, $c_2(u,x,\epsilon)$ are local coordinate on the space of leaves, hence any first integral is a function of them (depending on $\epsilon$).
On the other hand, any analytic germ $f(u,x,\epsilon)$ is uniquely decomposed as  $f=f_0(h,x,\epsilon)+u_1f_1(u_1,h,x,\epsilon)+u_2f_2(u_2,h,x,\epsilon)$, with $f_l$ analytic.
Writing $u_i=c_i(u,x,\epsilon)\cdot E_\chi^{-(-1)^i}(h,x,\epsilon)$ \eqref{eq:HC-E}, we see that for $f$ to be bounded when $x\to x_{i,\pm}$, we must have $f_i=0$, $i=1,2$.
Therefore $f=f_0$ which must then be independent of $x$.

A meromorphic function is a quotient of analytic ones.
\end{proof}

\begin{remark}
The statement remains true also if restricted to $\epsilon=0$, or a generic fixed $\epsilon$ (such that $\tfrac{\lambda^{(0)}(\epsilon)}{\epsilon}\notin\Z$).
\end{remark}

\begin{corollary}
The Lie algebra of analytic infinitesimal symmetries of $Z_{G,\epsilon}$  consists of Hamiltonian vector fields 
\begin{equation}\label{eq:HC-xi}
 \xi=a(u_1u_2,\epsilon)\big(u_1\partial_{u_1}-u_2\partial_{u_2}\big)=a(h,\epsilon)X_h,\qquad\textit{$a(h,\epsilon)$ analytic},
\end{equation}
and is commutative. It is also called the \emph{infinitesimal torus}.
\end{corollary}

The time-$1$ flow map of a vector field \eqref{eq:HC-xi} is given by
\begin{equation}\label{eq:HC-symmetry}
 u\mapsto \Cal T_a(u):=\Phi^1_{aX_h}(u)=
\left(\begin{smallmatrix} e^{a(h,\epsilon)} & 0 \\[3pt] 0 & e^{-a(h,\epsilon)}\end{smallmatrix}\right) u.
\end{equation}

\begin{definition}
A (transversely symplectic fibered)  \emph{isotropy} of the model vector field $Z_{G,\epsilon}$ is a germ of symplectic transformation
$(u,x,\epsilon)\mapsto(\phi(u,x,\epsilon),x,\epsilon)$ analytic in $u\in\sU$, such that $\phi^*Z_{G,\epsilon}=Z_{G,\epsilon}$.
An isotropy that is analytic in $x$ on a full neighborhood $\sX$ of both singularities will be called a \emph{symmetry}.
\end{definition}

\begin{definition}[Intersection sectors]\label{definition:HC-intersectionsectors}
For $\epsilon\in\sE_{\pm}\sminus\{0\}$ define the \emph{left and right intersection sectors}
$$\sX_{i,\pm}^\cap(\epsilon)=\{x\in\sX_{\pm}(\epsilon): x_{i,\pm}+e^{2\pi i}(x-x_{i,\pm})\in\sX_{\pm}(\epsilon)\}.$$
and for $\epsilon=0$ let $\sX_{i\pm}^\cap(0)$ be their limits.  
They are the domains of self-intersection of $\sX_\pm(\epsilon)$ attached to the points $x_{i,\pm}(\epsilon)$ \eqref{eq:HC-x12}.
\end{definition}

\begin{lemma}\label{lemma:HC-sectoralisotropy}
Let $\phi_{i,\pm}(u,x,\epsilon)$ be a sectoral isotropy of the normal form vector field $Z_{G,\epsilon}$, analytic and bounded for $x\in\sX_{i,\pm}^\cap(\epsilon)$, $u\in\sU$.
Then
\begin{equation}\label{eq:HC-sectoralisotropy}
c_i\circ \phi_{i,\pm}=c_i\cdot e^{f_i(h,c_i,\epsilon)},\quad\text{and}\quad  
c_j\circ \phi_{i,\pm}=c_j\cdot e^{f_j(h,\epsilon)}+g_j(h,c_i,\epsilon),
\end{equation}
for some analytic germs $f_i,f_j,g_j$.
\end{lemma}

\begin{proof}
The isotropy $\phi_{i,\pm}(u,x,\epsilon)$ is analytic in $u$ on some neighborhood of $u=0$ and bounded when $x\to x_{i,\pm}$. 
In particular, the restriction of $c\circ \phi_{i,\pm}$ to any fiber $\{x=cst\neq 0,\epsilon\}$ is analytic in $u$, and therefore, since $c\circ\phi_{i,\pm}(c,\epsilon)$ is 
independent of $x$, it is an analytic function of $c$ on some neighborhood of $c=0$.

We have $c_k=u_k E_\chi(h,x,\epsilon)^{(-1)^k}$, $k=1,2$, with $E_\chi$ given by \eqref{eq:HC-E}, and 
$$\lim_{\substack{x\to x_{i,\pm}\\ x\in\sX_\pm(\epsilon)}} E_\chi(h,x,\epsilon)^{(-1)^i}=0.$$ 
Writing $\phi_{i,\pm}=(\phi_{1,i,\pm},\phi_{2,i,\pm})$, its $k$-th component is given by 
$$\phi_{k,i,\pm}=\big(c_k\circ \phi_{i,\pm}\big)\cdot\big(E^{-(-1)^k}\circ h\circ \phi S_{i,\pm}\big),$$
and we see that the expansion of $c_i\circ \phi_{i,\pm}$ in powers of $c$ can contain only terms
$c_i^{n_i}c_j^{n_j}=u_i^{n_i}u_j^{n_j}E_\chi(h,x,\epsilon)^{(-1)^i(n_i-n_j)}$ with $n_i\geq n_j+1$, 
while the expansion of $c_j\circ \phi_{i,\pm}$ in powers of $c$ can contain only terms
$c_i^{n_i}c_j^{n_j}=u_i^{n_i}u_j^{n_j}E_\chi(h,x,\epsilon)^{(-1)^i(n_i-n_j)}$ with $n_i\geq n_j-1$.
Since $\phi_{i,\pm}$must be invertible $\det D_c(c\circ \phi_{i,\pm})\!\restriction_{c=0}\neq 0$ from which it follows that $c\circ \phi_{i,\pm}$ is of the form \eqref{eq:HC-sectoralisotropy}.
\end{proof}

Note that the hypersurface $\{u_i=0\}=\{c_i=0\}$ consists of all leaves of $Z_{G,\epsilon}$ that are bounded when $x\to x_{i,\pm}$ inside $X_{\pm}(\epsilon)$, and $\phi_{i,\pm}$ must preserve it.

\begin{proposition}
An isotropy of the normal form vector field $Z_{G,\epsilon}$ that is bounded and analytic on $\sU\times\sX_\pm(\epsilon)$ is a symmetry. It is given by a time-1 flow of some vector field \eqref{eq:HC-xi}.
\end{proposition}

\begin{proof}
An isotropy $\phi(u,x,\epsilon)$ of the model system  bounded and analytic on $\sU\times\sXE_\pm$ is in particular
bounded and analytic on $\sU\times\sX_{i,\pm}^\cap(\epsilon)$ for each $\epsilon\in\sE_\pm$, and
therefore by Lemma~\ref{lemma:HC-sectoralisotropy},
it is such that
$$c\circ\phi=\transp{(c_1e^{f_1(h,\epsilon)},c_2e^{f_2(h,\epsilon)})},\quad\text{i.e.}\quad \phi(u,x,\epsilon)=\transp{(u_1e^{f_1(h,\epsilon)},u_2e^{f_2(h,\epsilon)})}$$
for some analytic germs $f_1,f_2$.
The transverse symplecticity condition is then rewritten as
$$\tfrac{d}{dh}(h\circ\phi)=\tfrac{d}{dh}\big(he^{f_1(h,\epsilon)+f_2(h,\epsilon)}\big)=1,$$
which implies that $e^{f_1(h,\epsilon)+f_2(h,\epsilon)}=1$, i.e.
$\phi(u)=\left(\begin{smallmatrix} e^{f_1(h,\epsilon)} & 0 \\[3pt] 0 & e^{-f_1(h,\epsilon)}\end{smallmatrix}\right) u.$
\end{proof}

\begin{corollary}
The Lie group of (transversely symplectic fibered) symmetries \eqref{eq:HC-symmetry} of $Z_{G,\epsilon}$ is commutative and connected. 
It is called the \emph{exponential torus}.
\end{corollary}

A characterization of the Lie group of symmetries of a general system \eqref{eq:HC-system} will be given in Proposition~\ref{proposition:HC-symmetrygroup}.

\begin{corollary}\label{corollary:HC-21}
The normalizing transformations $\hat{\mbf\Psi}$ and $\mbf\Psi_\pm$ of Theorem~\ref{theorem:HC-normalization} are unique modulo composition with elements of the exponential torus (i.e. flow maps of infinitesimal symmetries analytic in $\epsilon$).
They are uniquely determined by the analytic germ $\psi^{(0)}(u,\epsilon)=\mbf\Psi_+(u,0,\epsilon)=\mbf\Psi_-(u,0,\epsilon)$,
cf. \eqref{eq:HC-hatPhi}, \eqref{eq:HC-Phi}.
\end{corollary}

\goodbreak

\subsection{Canonical general solutions}
The model system has a \emph{canonical general solution} $u(x,\epsilon;c)$ \eqref{eq:HC-u}, depending on an \emph{``initial condition'' parameter $c\in\C^2$}, uniquely determined by a choice of a branch of the function $E_\chi(h,x,\epsilon)$ \eqref{eq:HC-E}.
Correspondingly, $y(x,\epsilon;c)=\mbf\Psi_\pm(u(x,\epsilon;c),x,\epsilon)$ is a germ of general solution of the original system on $\sY\times\sX_\pm(\epsilon)$.
In order for this solution to have a continuous limit when $\epsilon\to 0$, one has to split the domain $\sX_\pm(\epsilon)$ in two parts, corresponding to the two parts of $\sX_\pm(0)$, by making a cut in between the singular points $x_{1,\pm}, x_{2,\pm}$ along a trajectory of \eqref{eq:HC-xvectfield} through the mid-point $\tfrac{\epsilon}{2}$ (see Figure~\ref{figure:HC-monodromy}).
Let us denote $\sX_\pm^{\uppie}(\epsilon)$ the upper and  $\sX_\pm^{\downpie}(\epsilon)$ the lower part (with respect to the oriented line $\lambda^{(0)}(0)\,\R$) of the cut domain 
$$ \sX_\pm(\epsilon)= \sX_\pm^{\uppie}(\epsilon)\cup \sX_\pm^{\downpie}(\epsilon).$$
The two parts of $\sX_\pm(\epsilon)$ intersect in the left and right intersection sectors $\sX_{i,\pm}^\cap(\epsilon)$ (Definition~\ref{definition:HC-intersectionsectors})
attached to $\{x_{1,\pm}$, $i=1,2$, and for $\epsilon\neq 0$ also in a central part along the cut. 


Now take two branches $E_\chi^{\uppie}(h,x,\epsilon)$ and $E_\chi^{\downpie}(h,x,\epsilon)$ of $E_\chi(h,x,\epsilon)$ on the two parts of the domain, that agree on the right 
intersection sector $\sX_{2,\pm}^\cap$, and have a limit when $\epsilon\to 0$. 
Correspondingly they determine a pair of general solutions of the model system
$$u^{\bullet}(x,\epsilon;c),\qquad \bullet=\uppie,\downpie,$$ 
and a pair of \emph{canonical general solutions} of the original system
\begin{equation}\label{eq:HC-y}
 y_\pm^{\bullet}(x,\epsilon;c):=\mbf\Psi_\pm(u^{\bullet}(x,\epsilon;c),x,\epsilon),\qquad \bullet=\uppie,\downpie.
\end{equation}
Since the transformation $\mbf\Psi_\pm$ is unique only modulo right composition with an exponential torus action 
$\Cal T_a(u,\epsilon)$ \eqref{eq:HC-symmetry}, which acts on  $u^{\bullet}(x,\epsilon;c)$ as
$$\Cal T_a(\cdot,\epsilon)\circ u^{\bullet}(x,\epsilon;c)=u^{\bullet}(x,\epsilon;\cdot)\circ \Cal T_a(c,\epsilon),$$
the solutions $y_\pm^{\bullet}$ are determined only up to the same right action of $\Cal T_a(c,\epsilon)$.

\subsection{Formal monodromy}
The formal monodromy operators are induced by monodromy acting on the solutions $u^\bullet(x,\epsilon;c)$, $\bullet=\uppie,\downpie$, of the model system. 
For $\epsilon\neq 0$ the induced action of formal monodromies along simple counterclockwise loops around each singular point $x_{i,\pm}=0,\epsilon$ on the 3 foliations is given by:
\begin{itemize}
\item Monodromy operators of the model system   
  \begin{equation}\label{eq:HC-mscrNa}
  \mscr N_{x_{i,\pm}}(\cdot,x,\epsilon)\circ u(x,\epsilon;c)= u(e^{2\pi i}(x-x_{i,\pm})+x_{i,\pm},\epsilon;c),
  \end{equation}
  acting on the foliation of the normal form vector field $Z_{G,\epsilon}$ \emph{commutatively} by 
\begin{equation}\label{eq:HC-mscrN}
\begin{aligned}
 \mscr N_{0}:\ u&\mapsto \exp\big(-2\pi i\tfrac{\chi^{0}(h,\epsilon)}{\epsilon}
\left(\begin{smallmatrix} 1 & 0 \\ 0 & -1 \end{smallmatrix}\right)\big)\cdot u=\Cal T_{-\frac{2\pi i}{\epsilon}\chi\restriction_{x=0}}(u),\\
\mscr N_{\epsilon}:\ u&\mapsto \exp\big(2\pi i[\tfrac{\chi^{0}(h,\epsilon)}{\epsilon}+\chi^{(1)}(h,\epsilon)]
\left(\begin{smallmatrix} 1 & 0 \\ 0 & -1 \end{smallmatrix}\right)\big)\cdot u=\Cal T_{\frac{2\pi i}{\epsilon}\chi\restriction_{x=\epsilon}}(u).
\end{aligned}
\end{equation}
The total monodromy of the model system  is given by
  \begin{equation*}
  \mscr N=\mscr N_{0}\circ \mscr N_{\epsilon}=\mscr N_{\epsilon}\circ \mscr N_{0}:\
u\mapsto \exp\big(2\pi i\chi^{(1)}(u_1u_2,\epsilon)
\left(\begin{smallmatrix} 1 & 0 \\ 0 & -1 \end{smallmatrix}\right)\big)\cdot u.
  \end{equation*}

\item Formal monodromy operators 
\begin{equation*}
  N_{x_{i,\pm}}(\cdot,\epsilon)\circ c(u,x,\epsilon)=c(\cdot,x,\epsilon)\circ \mscr N_{i,\pm}(u,x,\epsilon),
  \end{equation*}
acting on the space of initial conditions $c$ \emph{commutatively} by
\begin{equation}\label{eq:HC-N}
\begin{aligned}
 N_{0}:\ c&\mapsto \exp\big(-2\pi i\tfrac{\chi^{0}(h,\epsilon)}{\epsilon}
\left(\begin{smallmatrix} 1 & 0 \\ 0 & -1 \end{smallmatrix}\right)\big)\cdot c=\Cal T_{-\frac{2\pi i}{\epsilon}\chi\restriction_{x=0}}(c),\\
N_{\epsilon}:\ c&\mapsto\exp\big(2\pi i[\tfrac{\chi^{0}(h,\epsilon)}{\epsilon}+\chi^{(1)}(h,\epsilon)]
\left(\begin{smallmatrix} 1 & 0 \\ 0 & -1 \end{smallmatrix}\right)\big)\cdot c=\Cal T_{\frac{2\pi i}{\epsilon}\chi\restriction_{x=\epsilon}}(c),
\end{aligned}
\end{equation}
and a formal total monodromy
  \begin{equation*}
 N=N_{0}\circ N_{\epsilon}= N_{\epsilon}\circ N_{0}:\
c\mapsto \exp\big(2\pi i\chi^{(1)}(h,\epsilon)
\left(\begin{smallmatrix} 1 & 0 \\ 0 & -1 \end{smallmatrix}\right)\big)\cdot c.
  \end{equation*}

\item Formal monodromy operators $\mfr N_{{i,\pm}}(y,x,\epsilon)$ acting on the foliation of the original vector field $Z_{H,\epsilon}$:
  \begin{equation}\label{eq:HC-mfrN}
  \mfr N_{i,\pm}(\cdot,x,\epsilon)\circ \mbf\Psi_{\pm}(u,x,\epsilon)=\mbf\Psi_{\pm}(\cdot,x,\epsilon)\circ\mscr N_{i,\pm}(u,x,\epsilon),
  \end{equation}
and
  \begin{equation*}
  \mfr N_{\pm}(\cdot,x,\epsilon)=\mfr N_{1,\pm}(\cdot,x,\epsilon)\circ\mfr N_{2,\pm}(\cdot,x,\epsilon)=\mfr N_{2,\pm}(\cdot,x,\epsilon)\circ\mfr N_{1,\pm}(\cdot,x,\epsilon).
  \end{equation*}
\end{itemize}

The canonical solutions $u_\pm^{\uppie}$, $u_\pm^{\downpie}$ of the model system on the domains $\sX_\pm^{\uppie}$, $\sX_\pm^{\uppie}$, are defined such that they agree 
on the right intersection sector $\sX_{2,\pm}^\cap$. Therefore on the left intersection sector they are connected by the total formal monodromy operator 
$$ u_\pm^{\uppie}(x,\epsilon;c)=\mscr N(\cdot,x,\epsilon)\circ u_\pm^{\downpie}(x,\epsilon;c)=u_\pm^{\downpie}(x,\epsilon;\cdot) \circ N(c,\epsilon),\qquad x\in\sX_{1,\pm}^\cap,$$
and by the formal monodromy $\mscr N_{x_{i,\pm}}$ on the central cut between the two domains for $\epsilon\neq 0$  
(cf. Figure~\ref{figure:HC-monodromy}).

\subsection{Stokes operators and sectoral isotropies}
Let $y=\mbf\Psi_{\pm}(u,x,\epsilon)$ be the normalizing transformation on $\sX_\pm(\epsilon)$.
We call \emph{Stokes operators} the operators that change the determination of $\mbf\Psi_\pm$ over the left or right intersection sectors.
If $x\in\sX_{i,\pm}^\cap(\epsilon)$, then for $\epsilon\neq 0$ we denote
$$\bar x= e^{2\pi i}(x-x_{i,\pm})+x_{i,\pm}$$
the corresponding point in $\sX_{\pm}(\epsilon)$ on the other sheet, and extend this notation by limit to $\epsilon=0$. 
Namely,
$$\begin{tabular}{ll}
if $x\in\sX_{1,\pm}^\cap(\epsilon)\subset\sX_{\pm}^{\downpie}(\epsilon)$, & then $\bar x\in\sX_{\pm}^{\uppie}(\epsilon)$,\\
if $x\in\sX_{2,\pm}^\cap(\epsilon)\subset\sX_{\pm}^{\uppie}(\epsilon)$, & then $\bar x\in\sX_{\pm}^{\downpie}(\epsilon)$.
\end{tabular}$$
Then the Stokes operators are the operators
\begin{equation}
 \mbf\Psi_{\pm}(u,x,\epsilon)\mapsto \mbf\Psi_\pm(u,\bar x,\epsilon),\qquad x\in\sX_{i,\pm}^\cap(\epsilon),
\end{equation}
which for $\epsilon=0$ are the Stokes operators in the usual sense that send the Borel sum of the formal $x$-series $\hat{\mbf\Psi}(u,x,0)$ in one non-singular direction to the Borel sum in a following non-singular direction.

To each of these Stokes operators we associate \emph{sectoral isotropies} of the 3 foliations. 
\begin{itemize}
\item Sectoral isotropies $\mscr S_{i,\pm}(u,x,\epsilon)$ of the normal form vector field $Z_{G,\epsilon}$:
  \begin{equation}\label{eq:HC-mscrS}
  \mbf\Psi_{\pm}(\cdot,x,\epsilon)\circ \mscr S_{i,\pm}(u,x,\epsilon) =\mbf\Psi_\pm(u,\bar x,\epsilon),\qquad x\in\sX_{i,\pm}^\cap(\epsilon).
  \end{equation}
  The pair $(\mscr S_{1,\pm},\mscr S_{2,\pm})$ is an analog of the Martinet-Ramis invariant of saddle-node singularity \cite{MR1,RT}.
\item Sectoral isotropies $S_{1,\pm}(c,\epsilon)$ and $S_{2,\pm}(c,\epsilon)$ of the rectified vector field $Z_{0,\epsilon}=x(x-\epsilon)\partial_x$ in the $c$-space:
\begin{equation}\label{eq:HC-S}
 \begin{aligned}
  u^{\downpie}(x,\epsilon;\cdot)\circ S_{1,\pm}(c,\epsilon)&=\mscr S_{1,\pm}(\cdot,x,\epsilon)\circ u^{\downpie}(x,\epsilon;c),\qquad 
  x\in\sX_{1,\pm}^\cap(\epsilon),\\
  u^{\uppie}(x,\epsilon;\cdot)\circ S_{2,\pm}(c,\epsilon)&=\mscr S_{2,\pm}(\cdot,x,\epsilon)\circ u^{\uppie}(x,\epsilon;c),\qquad 
  x\in\sX_{2,\pm}^\cap(\epsilon).
  \end{aligned}
\end{equation}
  
\item Sectoral isotropies $\mfr S_{i,\pm}(y,x,\epsilon)$ of the original vector field $Z_{H,\epsilon}$:
  \begin{equation}\label{eq:HC-mfrS}
  \mfr S_{i,\pm}(\cdot,x,\epsilon)\circ \mbf\Psi_{\pm}(u,x,\epsilon)=\mbf\Psi_\pm(u,\bar x,\epsilon),\qquad x\in\sX_{i,\pm}^\cap(\epsilon).
  \end{equation}
\end{itemize}


\begin{figure}[t]
\centering
\includegraphics[width=.9\textwidth]{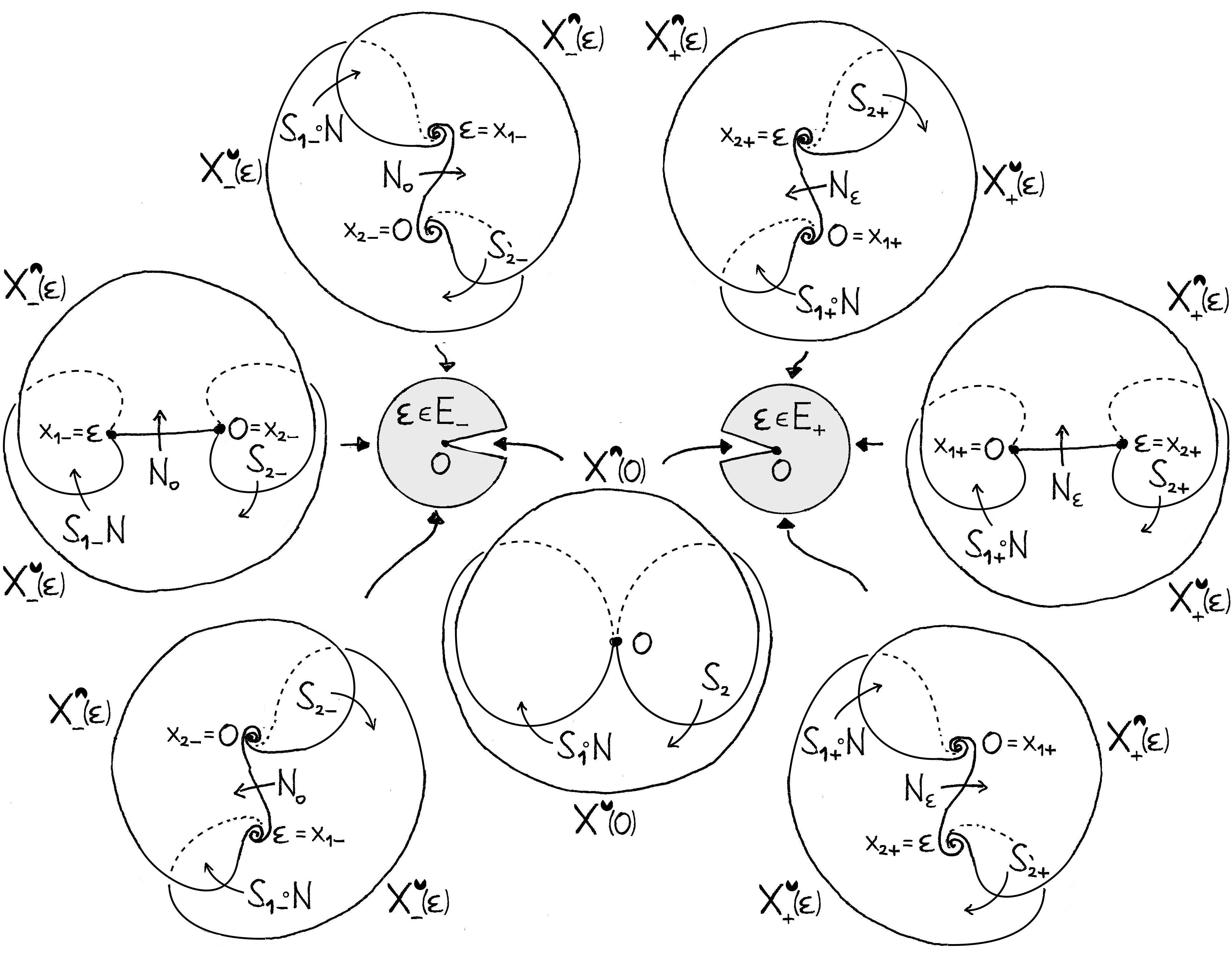}
\caption{Sectoral isotropies connecting the canonical general solutions $y_\pm^{\bullet}$ 
depending on $\epsilon\in\sE_\pm$.}
\label{figure:HC-monodromy}
\end{figure}

\begin{proposition}[Form of the Stokes isotropies]~\label{proposition:HC-Stokesisotropies}
Let $S_{i,\pm}(c,\epsilon)=\transp{(S_{1,i,\pm}(c,\epsilon),S_{2,i,\pm}(c,\epsilon))}$ be a Stokes sectoral isotropy \eqref{eq:HC-S}. Then
\begin{itemize}
 \item $S_{i,i,\pm}(c,\epsilon)=c_i+c_i^2\cdot\sigma_{i,i,\pm}(h,c_i,\epsilon)$ for an analytic germ $\sigma_{i,i,\pm}$, 
 \item $S_{j,i,\pm}(c,\epsilon)=c_j+\sigma_{j,i,\pm}(h,c_i,\epsilon)$ for an analytic germ $\sigma_{j,i,\pm}$, $j=3-i$,
\end{itemize}
subject to a condition $\det D_c(S_{i,\pm})=1$.
\end{proposition}
The term $\sigma_{j,i,\pm}(0,0)$ is responsible for the ramification of the ramified center manifold $y=\mbf\Psi_\pm(0,x,\epsilon)$ of the original vector field $Z_\epsilon$ at the sector $\sX_{i,\pm}^\cap(\epsilon)$.

\begin{proof}
The isotropy $\mscr S_{i,\pm}(u,x,\epsilon)$ is analytic in $u$ on some neighborhood of $u=0$ and bounded in $x$ with $\lim_{x\to x_{i,\pm}} \mscr S_{i,\pm}(u,x,\epsilon)=u$. 
By Lemma~\ref{lemma:HC-sectoralisotropy}, 
$$c_i\circ \mscr S_{i,\pm}=c_i\cdot e^{f_i(h,c_i,\epsilon)},\quad\text{and}\quad  c_j\circ \mscr S_{i,\pm}=c_j\cdot e^{f_j(h,\epsilon)}+g_j(h,c_i,\epsilon)$$ where $f_i,f_j,g_j$ are some analytic functions of $(h,c_i,\epsilon)$.

Knowing that $\lim_{x\to x_{i,\pm}} h\circ\mscr S_{i,\pm}(u,x,\epsilon)=h$,
$$h\circ \mscr S_{i,\pm}=h+c_i\cdot(\ldots),$$ 
where $(\ldots)$ is an analytic function of $(h,c_i,\epsilon)$, which implies that
$\lim_{x\to x_{i,\pm}}E_\chi^{-1}\cdot (E_\chi\circ h\circ \mscr S_{i,\pm})=1$.
Writing $\mscr S_{i,\pm}=(\mscr S_{1,i,\pm},\mscr S_{2,i,\pm})$, its $k$-th component is 
$$\mscr S_{k,i,\pm}=\big(c_k\circ \mscr S_{i,\pm}\big)\cdot\big(E^{-(-1)^k}\circ h\circ \mscr S_{i,\pm}\big).$$
We conclude that
$f_i=c_i\cdot(\ldots)$ and $f_j=0$.
\end{proof}

\subsection{Symmetry group of the system}

\begin{proposition}\label{proposition:HC-symmetrygroup}
 The group of (analytic transversely symplectic fibered) symmetries of a system \eqref{eq:HC-system} is either
\begin{enumerate}
 \item isomorphic to the exponential torus: this happens if and only if the system is analytically  equivalent to the model \eqref{eq:HC-normalform}, or
\item isomorphic to a finite cyclic group.
\end{enumerate}
If the symmetry group is non-trivial, then the system has an analytic center manifold (bounded analytic solution on a neighborhood of both singular points).
\end{proposition}

\begin{proof}
If $\Phi(y,x,\epsilon)$ is a symmetry of the system \eqref{eq:HC-system}, then 
$\Phi(\cdot,x,\epsilon)\circ\mbf\Psi_\pm(u,x,\epsilon)=\mbf\Psi_\pm(\cdot,x,\epsilon)\circ\phi(u,\epsilon)$ for some germ 
$$\phi:u\mapsto\left(\begin{smallmatrix} e^{a(h,\epsilon)} & \\ & e^{-a(h,\epsilon)}\end{smallmatrix}\right),$$
from the exponential torus,
and the analyticity of $\Phi$ means that this $\phi$ must commute with the Stokes operators $S_{i,\pm}$ \eqref{eq:HC-S} (note that $\phi$ acts the same way on $c$ as on $u$).
Using their characterization in Proposition~\ref{proposition:HC-Stokesisotropies}, this means that
$$\sigma_{i1,\pm}(h,c_1)=e^a\sigma_{i1,\pm}(h,e^a c_1),\qquad  \sigma_{i2,\pm}(h,c_2)=e^{-a}\sigma_{i2,\pm}(h,e^{-a} c_2),\quad i=1,2.$$
This can be satisfied only if
\begin{itemize}
 \item either $\sigma_{ij,\pm}(h,c_j)=0$ for all $i,j$, i.e. if $S_{1,\pm}=\id$, $S_{2,\pm}=\id$ and the system is analytically equivalent to its formal normal form,
\item or there is $k\in\N$ such that $c_j\sigma_{ij,\pm}(h,c_j)=0$ contains only powers of $c_j^k$ for all $i,j$, and $e^{ka}=1$, i.e. $a\in\frac{2\pi i}{k}\Z$.
\end{itemize}
\end{proof}

\subsection{Analytic classification}

\begin{definition}[Analytic invariants]
 The collection $(\chi,\{\mscr S_{1,+},\mscr S_{2,+},\mscr S_{1,-},\mscr S_{2,-}\})$ is called an \emph{analytic invariant} of a system \eqref{eq:HC-system}.
Two analytic invariants $(\chi,\{\mscr S_{i,\pm}\})$, $(\tilde\chi,\{\tilde{\mscr S}_{i,\pm}\})$ are equivalent if  
\begin{itemize}
 \item either $\chi=\tilde\chi$ \ and there is an element  $\phi(u,\epsilon)$ of the exponential torus, analytic in $\epsilon$, such that: 
$\ \mscr S_{i,\pm}=\phi\circ\tilde{\mscr S}_{i,\pm}\circ \mscr \phi^{\circ(-1)},\ i=1,2,$
\item or $\chi(h,x,\epsilon)=-\tilde\chi(-h,x,\epsilon)$ \ and there is an element  $\phi(u,\epsilon)$ of the exponential torus, analytic in $\epsilon$, such that: $\ \mscr S_{i,\pm}=J\phi\circ\tilde{\mscr S}_{j,\mp}\circ \mscr (J\phi)^{\circ(-1)},\ i=1,2,\ j=3-i,$
where $J:(u_1,u_2)\mapsto (u_2,-u_1)$.
Note that the definition of $\sE_\pm$, $\sX_{\pm}$ and $x_{i,\pm}$ depends on $\lambda(x,\epsilon)=\chi(0,x,\epsilon)$, therefore the relation $\tilde{\lambda}=-\lambda$ entails the renaming
$$\tilde\sE_\pm=\sE_\mp,\quad\tilde\sX_{\pm}^\bullet=\sX_\mp^\bullet,\quad \tilde x_{i,\pm}=x_{j,\mp}.$$
\end{itemize}
\end{definition}

By the construction, an \emph{analytic invariant} of a system \eqref{eq:HC-system} is uniquely defined up to the equivalence.

\begin{theorem}[Analytic classification]\label{theorem:HC-classification}
 Two systems \eqref{eq:HC-system} are analytically equivalent (in the sense of Definition~\ref{definition:HC-analequiv}) if and only if their analytic invariants are equivalent.
\end{theorem}

\begin{proof}
If $y=\Phi(\tilde y,x,\epsilon)$ is an analytic transformation from one system to another, then the sectoral normalizations $y=\mbf\Psi_{\pm}(u,x,\epsilon)$ and 
$\tilde y=\tilde{\mbf\Psi}_{\pm}(u,x,\epsilon)=\Phi\circ\mbf\Psi_{\pm}$ provide the same analytic invariant. 
Conversely, if the analytic invariants are equivalent, then up to modifying one of the normalizing transformation, one can suppose that they are in fact equal, in which case
$\Phi_\pm=\tilde{\mbf\Psi}_{\pm}\circ\mbf\Psi_{\pm}^{\circ(-1)}$ are analytic transformations between the systems on $\sE_+$ and $\sE_-$. In fact $\Phi_+=\Phi_-$ is an analytic on the whole $\epsilon$-neighborhood $\sE$. 
Indeed, the composition $\Phi_+\circ\Phi_-^{\circ-1}$ is a symmetry of the second system on the intersection $\sE_+\cap\sE_-$, and as such it is determined by its value at $x=0$; but since $\tilde{\mbf\Psi}_{+}\!\restriction_{x=0}=\tilde{\mbf\Psi}_{-}\!\restriction_{x=0}=\tilde{\psi}^{(0)}(u,\epsilon)$ and
${\mbf\Psi}_{+}\!\restriction_{x=0}={\mbf\Psi}_{-}\!\restriction_{x=0}={\psi}^{(0)}(u,\epsilon)$ are analytic in $\epsilon$ \eqref{eq:HC-Phi},
this means that $\Phi_+\circ\Phi_-^{\circ-1}\!\restriction_{x=0}=\id$ and therefore $\Phi_+\circ\Phi_-^{\circ-1}=\id$.
\end{proof}

\subsection{Decomposition of monodromy operators}

For $\epsilon\neq 0$, let $x_0\in\sX_\pm(\epsilon)\sminus\{0,\epsilon\}$ be a base-point, and let two counterclockwise simple loops around the singular points $x_{i,\pm}$, $i=1,2$, be as in Figure~\ref{figure:HC-loops}.
Correspondingly, we have two monodromy operators $\mfr M_{x_{i,\pm}}$ acting on the foliation by the solutions of the original system \eqref{eq:HC-system} by analytic continuation 
along the loops.
Since the monodromy operators $\mfr M_{x_{i,\pm}}$ act on the foliation, they  are independent of the choice of the two-parameter general solution on which they act on the left 
(a different general solution is related to it by a change of the parameter, independent of $x$ and acting on the right). In particular
$$\mfr M_{x_{1,+}}\!=\mfr M_{x_{2,-}},\quad \mfr M_{x_{2,+}}\!=\mfr M_{x_{1,-}}.$$

\begin{figure}[t]
\centering 
\includegraphics[width=.25\textwidth]{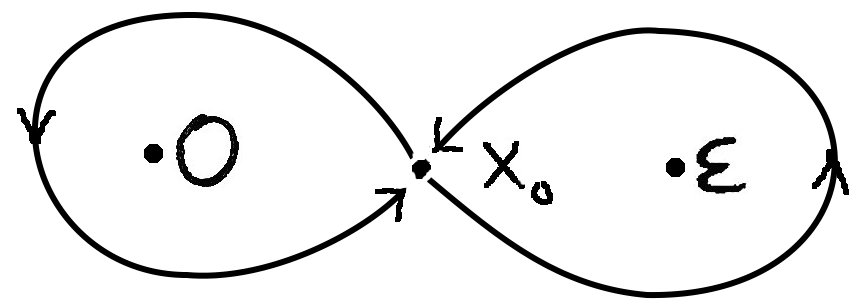}
\caption{Simple loops along which the monodromy operators $\mfr M_{0}$ and $\mfr M_{\epsilon}$ are defined.}
\label{figure:HC-loops}
\end{figure}

\begin{theorem}\label{proposition:HC-monodromy}\label{theorem:HC-decomposition}
 For $\epsilon\neq 0$, the monodromy operators $\mfr M_{x_{i,\pm}}$ of the original foliation are well defined on some open neighborhood of the ramified center manifold $y=\mbf\Psi_\pm(0,x,\epsilon)$ in $\sY\times\sX_\pm(\epsilon)$.
Their (left) action is given by
\begin{equation}
\mfr M_{x_{i,\pm}}=\mfr S_{i,\pm}\circ\mfr N_{i,\pm},
\end{equation}
where $\mfr S_{i,\pm}$ are the Stokes operators \eqref{eq:HC-mfrS} and  $\mfr N_{i,\pm}$ are the formal monodromy operators \eqref{eq:HC-mfrN}. 
Hence
\begin{align*}
\mfr M_{0}&=\mfr S_{1,+}\circ\mfr N_{1,+}=\mfr S_{2,-}\circ\mfr N_{2,-},\\
\mfr M_{\epsilon}&=\mfr S_{2,+}\circ\mfr N_{2,+}=\mfr S_{1,-}\circ\mfr N_{1,-}.
\end{align*}
Their right action on analytic extension of the canonical general solutions $y_\pm^\bullet$ \eqref{eq:HC-y} to the whole $\sX_\pm(\epsilon)$ is given by
\begin{align*}
 \mfr M_{x_{i,\pm}}(y_\pm^{\bullet}(x,\epsilon;c),x,\epsilon)=y_\pm^{\bullet}(x,\epsilon;\cdot)\circ M_{i,\pm}^{\bullet}(c,\epsilon),\qquad\bullet=\uppie,\downpie,
\end{align*}
where
\begin{equation}\label{eq:HC-Mdecomposition}
\begin{aligned}
M_{1,\pm}^{\uppie}&=N_{x_{2,\pm}}^{\circ(-1)}\circ S_{1,\pm}\circ N,&  M_{1,\pm}^{\downpie}&=S_{1,\pm}\circ N_{x_{1,\pm}},\\
M_{2,\pm}^{\uppie}&=S_{2,\pm}\circ N_{x_{2,\pm}}, &  M_{2,\pm}^{\downpie}&=N_{x_{2,\pm}}\circ S_{2,\pm},
\end{aligned}
\end{equation}
cf. Figure~\ref{figure:HC-monodromy}.
\end{theorem}

\begin{proof}[Proof of Theorem~\ref{proposition:HC-monodromy}]

The demonstration of the given formulas is straightforward, but we include it here for the sake of completeness. 
Let $\bar x:=x_{i,\pm}+e^{2\pi i}(x-x_{i,\pm})$, then 
\begin{align*}
\mfr M_{x_{i,\pm}}(\cdot,x,\epsilon)\circ y_\pm(x,\epsilon;c)&=y_\pm(\bar x,\epsilon;c)=\mbf\Psi_\pm(\cdot,\bar x,\epsilon)\circ u(\bar x,\epsilon;c)\\
&=\big(\mfr S_{i,\pm}\circ \mbf\Psi_\pm\big)(\cdot,x,\epsilon)\circ u(\bar x,\epsilon;c)\\
&=\big(\mfr S_{i,\pm}\circ \mbf\Psi_\pm\circ\mscr N_{i,\pm}\big)(\cdot,x,\epsilon)\circ u(x,\epsilon;c)\\
&=\big(\mfr S_{i,\pm}\circ\mfr N_{i,\pm}\circ \mbf\Psi_\pm\big)(\cdot,x,\epsilon)\circ u(x,\epsilon;c)\\
&=\big(\mfr S_{i,\pm}\circ\mfr N_{i,\pm}\circ y_\pm\big)(x,\epsilon;c),
\end{align*}
using \eqref{eq:HC-mfrS}, \eqref{eq:HC-mscrNa}, \eqref{eq:HC-mfrN}.
Similarly, to calculate $M_{1,\pm}^{\uppie}$ for example,
\begin{align*}
y_\pm^{\uppie}(\bar x,\epsilon;c)&=\mbf\Psi_\pm(\cdot,\bar x,\epsilon)\circ u^{\uppie}(\bar x,\epsilon;c)\\
&=\mbf\Psi_\pm(\cdot,\bar x,\epsilon)\circ u^{\downpie}(\bar x,\epsilon;\cdot)\circ N_{x_{2,\pm}}(c,\epsilon)\\
&=\mbf\Psi_\pm(\cdot,\bar x,\epsilon)\circ u^{\downpie}(x,\epsilon;\cdot)\circ N_{x_{1,\pm}} \circ N_{x_{2,\pm}}(c,\epsilon)\\
&=\mbf\Psi_\pm(\cdot,x,\epsilon)\circ u^{\downpie}(x,\epsilon;\cdot)\circ S_{1,\pm} \circ N_{x_{1,\pm}} \circ N_{x_{2,\pm}}(c,\epsilon)\\
&=\mbf\Psi_\pm(\cdot,x,\epsilon)\circ u^{\uppie}(x,\epsilon;\cdot) \circ N_{x_{2,\pm}}^{\circ(-1)} \circ S_{1,\pm} \circ N_{x_{1,\pm}} \circ N_{x_{2,\pm}}(c,\epsilon)\\
&=y_\pm^{\uppie}(x,\epsilon;\cdot)\circ N_{x_{2,\pm}}^{\circ(-1)} \circ S_{1,\pm} \circ N(c,\epsilon),
\end{align*}
see Figure~\ref{figure:HC-monodromy}.
The general solutions $y_\pm^{\uppie},\ y_\pm^{\downpie}$ are related by
$$y_\pm^{\uppie}(x,\epsilon;c)=y_\pm^{\downpie}(x,\epsilon;\cdot)\circ N_{x_{2,\pm}}(c,\epsilon).$$
\end{proof}

Note that in general, a composition of the two monodromies may not be defined if the image of the first does not intersect the domain of definition the second.

\subsection{Accumulation of monodromy}

\begin{definition}[Monodromy pseudogroup]~\\[-12pt]
\begin{enumerate}
\item  For $\epsilon\neq 0$, the pseudogroup generated by the monodromy operators 
$$\langle\mfr M_{0}(\cdot,x,\epsilon),\ \mfr M_{\epsilon}(\cdot,x,\epsilon)\rangle$$ 
is called the (local) \emph{monodromy pseudogroup}.
The pseudogroup generated by the corresponding action on the initial condition $c$
$$\langle M_{1,\pm}^\bullet(\cdot,\epsilon),\ M_{2,\pm}^\bullet(\cdot,\epsilon) \rangle$$ 
is its representation with respect to the general solution $y_\pm^\bullet(x,\epsilon;c)$.
\item For $\epsilon= 0$, the pseudogroup generated by the Stokes operators and by the 
elements of the exponential torus (pushed-forward by the sectoral transformations $\mbf\Psi^{\bullet}$):  
$$\langle \mfr S_{1,\pm}(\cdot,x,0),\ \mfr S_{2,\pm}(\cdot,x,0),\ \{\mfr T_a^\bullet(\cdot,x,0)\}_a \rangle,$$ 
where
$\mfr T_a^\bullet(\cdot,x,0)=\mbf\Psi_\pm^{\bullet}(\cdot,x,0)_* \Cal T_{a}=\Phi^1_{(\mbf\Psi_\pm^{\bullet})_*(a(h)X_h)}$,
is called the (local) \emph{wild monodromy pseudogroup}.
The pseudogroup generated by the corresponding action on the initial condition $c$
\begin{equation}\label{eq:HC-wildmonodromy}
 \langle S_{1,\pm}(\cdot,0),\ S_{2,\pm}(\cdot,0),\ \{\Cal T_a(\cdot)\}_a\rangle,
\end{equation}
is its representation with respect to the formal transseries solution $\hat y(x,0;c)=\hat{\mbf\Psi}(u(x,0;c),x,0)$.
\end{enumerate}
\end{definition}
Note that the pseudogroup \eqref{eq:HC-wildmonodromy} is independent of the freedom of choice of the sectoral normalizations $\mbf\Psi_\pm^{\bullet}$ of Theorem~\ref{theorem:HC-normalization0}.

One of the main goals of this paper is to understand the relation between the monodromy pseudogroup for $\epsilon\neq 0$ and the wild monodromy pseudogroup for $\epsilon=0$.

\medskip

\textbf{Suppose} that the formal invariant $\chi(h,x,\epsilon)$ is such that
\begin{equation}\label{eq:HC-assumption}
 \chi(h,x,0)=\lambda^{(0)}(0)+x\chi^{(1)}(h,0),
\end{equation}
and therefore
$$\chi(h,x,\epsilon)=\lambda^{(0)}(0)+\epsilon\tfrac{\partial\chi^{(0)}}{\partial\epsilon}(h,0)+x\chi^{(1)}(h,0)+O(x\epsilon)+O(\epsilon^2).$$
Let $\{\epsilon_n\}_{n\in\pm\N}$ be sequence in $\sE_\pm\sminus\{0\}$ defined by
\begin{equation}\label{eq:HC-epsilon_n}
 \tfrac{\lambda^{(0)}(0)}{\epsilon_n}=\tfrac{\lambda^{(0)}(0)}{\epsilon_0}+n,\qquad \epsilon_0\in\sE_\pm\sminus\{0\}
\end{equation}
along which the exponential factor $e^{\frac{2\pi i\lambda^{(0)}}{\epsilon}}$ in the formal monodromy \eqref{eq:HC-N} stays constant, and denote
\begin{equation}\label{eq:HC-kappa}
 \kappa:=e^{\frac{2\pi i\lambda^{(0)}}{\epsilon_0}}.
\end{equation}
Then the formal monodromy operators $\mscr N_0(u,x,\epsilon),\mscr N_\epsilon(u,x,\epsilon)$, resp. $N_0(u,x,\epsilon), N_\epsilon(u,x,\epsilon)$, converge along each such sequence to a symmetry of the model system (element of the exponential torus) 
\begin{equation}\label{eq:HC-tildeN}
\begin{aligned}
\hskip-6pt\Tilde{N}_0(\kappa;u,x):= \!\!\lim_{n\to\pm\infty\!\!}\!\! N_0(u,x,\epsilon_n)&:\ c\mapsto 
\left(\!\!\begin{smallmatrix} \kappa^{-1}e^{-2\pi i\frac{\partial\chi^{(0)}}{\partial\epsilon}(h,0)} & 0 \\ 0 & \kappa e^{2\pi i\frac{\partial\chi^{(0)}}{\partial\epsilon}(h,0)} \end{smallmatrix}\!\right) c,\\[3pt]
\hskip-6pt\Tilde{N}_\epsilon(\kappa;u,x):= \!\!\lim_{n\to\pm\infty\!\!}\!\! N_\epsilon(u,x,\epsilon_n)&:\ c\mapsto
\left(\!\!\begin{smallmatrix} \kappa e^{2\pi i[\frac{\partial\chi^{(0)}}{\partial\epsilon}(h,0)+\chi^{(1)}(h,0)]}\hskip-15pt & 0 \\ 0 & \hskip-15pt\kappa^{-1} e^{-2\pi i[\frac{\partial\chi^{(0)}}{\partial\epsilon}(h,0)+\chi^{(1)}(h,0)]} \end{smallmatrix}\!\right) c,\hskip-6pt
\end{aligned}
\end{equation}
$\kappa\in\C^*$.
This implies that also the monodromy operators $M_{i,\pm}^\cdot(c,\epsilon)$, resp. $\mfr M_{x_{i,\pm}}(y,x,\epsilon)$, converge along such sequences 
$\{\epsilon_n\}_{n\in\pm\N}\subset \sE_\pm\sminus\{0\}$.
Denote 
$$\Tilde{\mfr N}_{{i,\pm}}^\bullet(\kappa;y,x):= \lim_{n\to\pm\infty} \mfr N_{{i,\pm}}(y,x,\epsilon_n),\qquad x\in\sX_\pm^\bullet,\quad\bullet=\uppie,\downpie.$$

\begin{theorem}\label{theorem:HC-accumulation}
Suppose that the formal invariant of the form \eqref{eq:HC-assumption}.
Then the monodromy operators of the system \eqref{eq:HC-system} for $\epsilon\neq 0$ accumulate along the sequences $\{\epsilon_n\}_{n\in\pm\N}$  \eqref{eq:HC-epsilon_n}
to a 1-parameter family of wild monodromy operators
\begin{align*}
\Tilde{\mfr M}_{1,\pm}(\kappa;y,x):=\lim_{n\to\pm\infty} \mfr M_{x_{1,\pm}}(y,x,\epsilon_n)&=\mfr S_{1,\pm}(\cdot,x,0)\circ\tilde{\mfr N}_{x_{1,\pm}}^{\downpie}(\kappa;y,x),\qquad 
x\in\sX^\cap_{1,\pm}(0),\\
\Tilde{\mfr M}_{2,\pm}(\kappa;y,x):=\lim_{n\to\pm\infty} \mfr M_{x_{2,\pm}}(y,x,\epsilon_n)&=\mfr S_{2,\pm}(\cdot,x,0)\circ\tilde{\mfr N}_{x_{2,\pm}}^{\uppie}(\kappa;y,x),\qquad 
x\in\sX^\cap_{2,\pm}(0).
\end{align*}
In particular,
if we replace $\kappa$ by $e^{-2\pi i\big[\frac{\partial\chi^{(0)}}{\partial\epsilon}(h,0)+\delta_{i,\pm}(0)\chi^{(1)}(h,0)\big]}$, $\delta_{i,\pm}(\epsilon)=\frac{x_{i,\pm}(\epsilon)}{\epsilon}$, so that $\tilde{\mfr N}_{x_{i,\pm}}^{\bullet}(\kappa;y,x)$ becomes an identity, we obtain the Stokes operators
\begin{equation}
\mfr S_{i,\pm}(y,x,0)=\Tilde{\mfr M}_{i,\pm}(e^{-2\pi i\big[\frac{\partial\chi^{(0)}}{\partial\epsilon}(h,0)+\delta_{i,\pm}(0)\chi^{(1)}(h,0)\big]};y,x).
\end{equation}
The vector field
\begin{equation}\label{eq:HC-infinitesimalgenerator}
\dot y=\pm(-1)^{i}\big(\kappa\tfrac{\partial}{\partial\kappa}\tilde{\mfr M}_{i,\pm}(\kappa;y,x)^{\circ(-1)}\big)\circ \tilde{\mfr M}_{i,\pm}(\kappa;y,x)
\end{equation}
equals to the push-forward $\mbf\Psi_\pm^\bullet(\cdot,x,0)_* \big(X_h\big)$
of the vector field $X_h=u_1\partial_{u_1}-u_2\partial_{u_2}$, where $\bullet=\downpie$ if $i=1$ and $\bullet=\uppie$ if $i=2$, 
which ``generates'' the commutative Lie algebra of bounded infinitesimal symmetries on the sector $\sX^\bullet(0)$.
\end{theorem}

\begin{conclusion}
The knowledge of the limits $\Tilde{\mfr M}_{i,\pm}(\kappa;y,x)$, $\kappa\in\C^*$, allows to recover the infinitesimal symmetry $\mbf\Psi_\pm^\bullet(\cdot,x,0)_*\big(X_h\big)$
\eqref{eq:HC-infinitesimalgenerator},
and hence its Hamiltonian, the bounded first integral $\mbf\Psi_\pm^\bullet(\cdot,x,0)_*(h)$ which vanishes at the singular points $(y_0(x_{i,\pm},\epsilon),x_{i,\pm})$, and therefore, 
knowing the formal invariant $\chi$, also the formal monodromy operators $\tilde{\mfr N}_{x_{1,\pm}}^{\bullet}(\kappa;y,x)$,
and finally the Stokes isotropies $\mfr S_{i,\pm}(y,x,0)$.
\end{conclusion}

\begin{remark}
In the case when the assumption \eqref{eq:HC-assumption} is not met, one can nevertheless get a similar ``accumulation'' result by replacing in \eqref{eq:HC-epsilon_n} 
$\lambda^{(0)}(0)$ by $\chi(h,0,0)$ and defining $\kappa(h)=e^{\frac{2\pi i\chi(h,0,0)}{\epsilon_0(h)}}$ \eqref{eq:HC-kappa}.
\end{remark}

\section{Confluence in $2\!\times\!2$ traceless linear systems and their differential Galois group}\label{section:HC-linear}

To illustrate the matter of the previous section, let us consider a confluence of two regular singular points to a non-resonant irregular singular point in a family of linear systems
\begin{equation}\label{eq:HC-wm-1}
x(x-\epsilon)\frac{dy}{dx}=A(x,\epsilon)y,\qquad y\in\C^2,
\end{equation}
where $A$ is a $2\!\times\! 2$ traceless complex matrix depending analytically on $(x,\epsilon)\in(\C\!\times\!\C,0)$, such that 
$A(0,0)\neq 0$ has two distinct eigenvalues $\pm\lambda^{(0)}(0)$. 

The Theorem~\ref{theorem:HC-normalization} in this case can be found in the thesis of Parise \cite{Pa} and in the work of Lambert and Rousseau \cite{LR} (see also \cite{HLR}).
It provides us with a canonical fundamental solution matrices
\begin{equation}
 Y_\pm^\bullet(x,\epsilon)=\mbf\Psi_\pm(x,\epsilon)\cdot U_\pm^\bullet(x,\epsilon),\qquad\bullet=\uppie,\downpie,
\end{equation}
where the transformation matrix $\mbf\Psi_\pm(x,\epsilon)$ is bounded on $\sXE_\pm$, and
\begin{equation}
 U_\pm^\bullet(x,\epsilon)=\left(\begin{smallmatrix} E_\lambda^\bullet(x,\epsilon) & 0 \\ 0 & E_\lambda^\bullet(x,\epsilon)^{-1} \end{smallmatrix}\right),
\end{equation}
is a solution to the diagonal model system
\begin{equation*}
x(x-\epsilon)\frac{dy}{dx}=\left(\begin{smallmatrix} \lambda(x,\epsilon) &  \\  & -\lambda(x,\epsilon) \end{smallmatrix}\right)y.
\end{equation*}
The solution basis $Y_\pm^\bullet(x,\epsilon)$ is also called a \emph{mixed basis}: the first (resp. second) column spans the subspace of solutions that asymptotically vanish when $x\to x_{1,\pm}(\epsilon)$ (resp. when $x\to x_{2,\pm}(\epsilon)$), and it is an eigensolution with respect to the corresponding monodromy operator $\mfr M_{x_{1,\pm}}$ (resp. $\mfr M_{x_{2,\pm}}$) associated to its eigenvalue $e^{\pm2\pi i\frac{\lambda(x_{1,\pm},\epsilon)}{\epsilon}}$ (resp. $e^{\pm2\pi i\frac{\lambda(x_{2,\pm},\epsilon)}{\epsilon}}$). 
A general solution is a linear combination
$$y_\pm^{\bullet}(x,\epsilon;c)=Y_\pm^{\bullet}(x,\epsilon)\cdot c,\qquad\bullet=\uppie,\downpie.$$

Let $\K$ be the field of meromorphic functions of the variable $x$ on a fixed small neighborhood of $0$, equipped with the differentiation $\frac{d}{dx}$.
For a fixed small $\epsilon$, the local \emph{differential Galois group} (also called the \emph{Picard-Vessiot group}) of the system \eqref{eq:HC-wm-1} is the group of $\K$-automorphisms of the differential field
$\K\langle Y(\cdot,\epsilon)\rangle$, generated by the components of any fundamental matrix solution $Y(x,\epsilon)$.
The differential Galois group acts on the foliation associated to the system by left multiplication.
Fixing a fundamental solution matrix $Y=Y_\pm^{\bullet}$, then each automorphism is represented by a right multiplication of $Y_\pm^{\bullet}$ by a constant invertible matrix,
hence the differential Galois group is represented by an (algebraic) subgroup of $\SL_2(\C)$ acting on the right.

It is well known \cite{MR2,SP} that the differential Galois group is the Zariski closure of 

\begin{itemize}\leftskip12pt \itemsep0pt 
\item[$\epsilon\neq 0$:] \emph{the monodromy group} generated by the two monodromy operators around the singular points $0$ and $\epsilon$, 
\item[$\epsilon= 0$:] \emph{the wild monodromy group}\footnote{The name ``wild monodromy''is borrowed from \cite{MR2}.} 
generated by the Stokes operators and the linear exponential torus\,\footnote{For general linear systems one would need to add also the total formal monodromy $N(0)=T_{2\pi i\lambda^{(1)}(0)}$, which in our case already belongs to the exponential torus.} 
which acts on the fundamental solutions $Y_\pm^{\bullet}$ as
\begin{equation}
 \mfr T_a^\bullet:Y_\pm^{\bullet}(x,0)\mapsto Y_\pm^{\bullet}(x,0)\cdot T_a,\quad\text{where}\quad 
T_a=\left(\begin{smallmatrix} e^a & \\ & e^{-a} \end{smallmatrix}\right),\quad a\in\C.
\end{equation}
\end{itemize} 
The question is how are these two different descriptions related?

\medskip

The monodromy matrices of $Y_\pm^{\uppie},\ Y_\pm^{\downpie}$, around the points $x_{1,\pm}(\epsilon)$, $x_{2,\pm}(\epsilon)$, $\epsilon\in\sE_\pm\sminus\{0\}$, are given respectively by
\begin{align*}
M_{1,\pm}^{\uppie}&=N_{x_{2,\pm}}^{(-1)} S_{1,\pm} N,& M_{2,\pm}^{\uppie}&=S_{2,\pm} N_{x_{2,\pm}}\\
M_{1,\pm}^{\downpie}&=S_{1,\pm} N_{x_{1,\pm}}, & M_{2,\pm}^{\downpie}&=N_{x_{2,\pm}} S_{2,\pm},
\end{align*}
where $S_{i,\pm}$ are of the form
\begin{equation*}
S_{1,\pm}=\left(\begin{smallmatrix} 1 & 0 \\[3pt] s_{1,\pm} & 1 \end{smallmatrix}\right),\qquad
S_{2,\pm}=\left(\begin{smallmatrix} 1 & s_{2,\pm} \\[3pt] 0 & 1 \end{smallmatrix}\right).
\end{equation*}
In particular $M_{1,\pm}^{\bullet}$ is lower-triangular and $M_{2,\pm}^{\bullet}$ is upper-triangular. 

When $\epsilon\to 0$ along a sequence\footnote{The idea of taking limits of monodromy along such sequences can be found in the works of J.-P.~Ramis \cite{Ra1} or A.~Duval \cite{Du}.}
 $\frac{1}{\epsilon_n}=\frac{1}{\epsilon_0}+\frac{n}{\lambda^{(0)}}$, $n\in\pm\N$, $\epsilon_0\in\sE_\pm\sminus\{0\}$,
these monodromy converge respectively to $\tilde M_{i,\pm}^{\bullet}(\kappa)=\lim_{n\to\pm\infty} M_{i,\pm}^{\bullet}(\epsilon_n)$ given by
\begin{align*}
\tilde M_{1,+}^{\uppie}(\kappa)&=\tilde N_\epsilon(\kappa)^{-1}S_{1}(0)N(0),& 
\tilde M_{1,+}^{\downpie}(\kappa)&=S_{1}(0)\tilde N_0(\kappa),\\
\tilde M_{2,+}^{\uppie}(\kappa)&=S_{2}(0)\tilde N_\epsilon(\kappa),&
\tilde M_{2,+}^{\downpie}(\kappa)&=\tilde N_\epsilon(\kappa)S_{2}(0),\\[6pt]
\tilde M_{1,-}^{\uppie}(\kappa)&=\tilde N_0(\kappa)^{-1}S_{1}(0)N(0),& 
\tilde M_{1,-}^{\downpie}(\kappa)&=S_{1}(0)\tilde N_\epsilon(\kappa),\\ 
\tilde M_{2,-}^{\uppie}(\kappa)&=S_{2}(0)\tilde N_0(\kappa),&
\tilde M_{2,-}^{\downpie}(\kappa)&=\tilde N_0(\kappa)S_{2}(0),
\end{align*}
with 
$$\tilde N_0(\kappa)=\left(\begin{smallmatrix} \kappa^{-1} & \\[3pt] & \kappa \end{smallmatrix}\right)T_{-2\pi i \frac{d\lambda^{(0)}}{d\epsilon}(0)},\quad
\tilde N_\epsilon(\kappa)=\left(\begin{smallmatrix} \kappa & \\[3pt] & \kappa^{-1} \end{smallmatrix}\right)T_{2\pi i [\frac{d\lambda^{(0)}}{d\epsilon}(0)+\lambda^{(1)}(0)]},\quad
N(0)=T_{2\pi i\lambda^{(1)}(0)}.$$
We call them \emph{wild monodromy matrices}. 
The family of them 
$$\{\tilde M_{1,\pm}^\bullet(\kappa),\ \tilde M_{2,\pm}^\bullet(\kappa)\ |\ \kappa\in\C^*\}$$ 
generates the same group, the representation of the \emph{wild monodromy group} with respect to the formal solution $\hat Y(x,\epsilon)$, as does the collection of 
the Stokes matrices and the linear exponential torus
$$\{S_{1}(0),S_{2}(0)\}\cup\{T_a\ |\ a\in\C \}.$$ 
Hence we have the following theorem, whose general idea was suggested by J.-P.~Ramis \cite{Ra1}:

\begin{theorem}
When $\epsilon\to 0$ the elements of the monodromy group of the system \eqref{eq:HC-wm-1} accumulate to generators of the wild monodromy group of the limit system.
\end{theorem}


\section{Confluent degeneration of the sixth Painlevé equation to the fifth}\label{section:HC-PV}

The sixth Painlevé equation is
\begin{equation*}\label{eq:HC-PVI}
\begin{split}
P_{VI}:\ \ 
q''&=\frac{1}{2}\Big(\frac{1}{q}+\frac{1}{q-1}+\frac{1}{q-t}\Big)(q')^2-
\Big(\frac{1}{t}+\frac{1}{t-1}+\frac{1}{q-t}\Big)q'\\
&\ +\frac{q(q-1)(q-t)}{2\,t^2(t-1)^2}
\Big[(\vartheta_\infty\!-\!1)^2-\vartheta_0^2\frac{t}{q^2}+\vartheta_1^2\frac{(t-1)}{(q-1)^2}+(1\!-\!\vartheta_t^2)\frac{t(t-1)}{(q-t)^2}\Big],
\end{split}
\end{equation*}
where $\vartheta=(\vartheta_0,\vartheta_t,\vartheta_1,\vartheta_\infty)\in\C^4$ are complex constants.
It is  a reduction to the $q$-variable of a time dependent Hamiltonian system \cite{Oka1}
\begin{align}\label{eq:HC-PVIhamiltonian}
\frac{dq}{dt}=\, \frac{\partial H_{VI}(q,p,t)\!}{\partial p},\qquad
\frac{dp}{dt}=-\frac{\partial H_{VI}(q,p,t)}{\partial q},
\end{align}
with a polynomial Hamiltonian function 
\begin{equation*}
\begin{split}
H_{VI}=\tfrac{1}{t(t-1)}\Big[&
q(q-1)(q-t)p^2-\Big(\vartheta_0(q-1)(q-t)+\vartheta_1 q(q-t)+(\vartheta_t-1)q(q-1) \Big)p \\
& + \tfrac{(\vartheta_0+\vartheta_1+\vartheta_t-1)^2-\vartheta_\infty^2}{4}(q-t)\Big].
\end{split}
\end{equation*}
It has three simple (regular) singular points on the Riemann sphere $\CP^1$ at $t=0,1,\infty$.

The fifth Painlevé equation $P_{V}$\,\footnote{The equation \eqref{eq:HC-PV} is the fifth equation of Painlevé with a parameter $\eta_1=-1$. A general form of this equation would be obtained by a further change of variable $\tilde t\mapsto-\eta_1 \tilde t$. The degenerate case $P_V^{deg}$ with $\eta_1=0$ which has only a regular singular point at $\infty$ is not considered here.} 
\begin{equation*}\label{eq:HC-PV}
P_{V}:\ \ 
q''=\Big(\frac{1}{2q}+\frac{1}{q\!-\!1}\Big)(q')^2-
\frac{1}{\tilde t}q'+
\frac{(q\!-\!1)^2}{2\tilde t^2}\Big((\vartheta_\infty\!-1)^2 q-\frac{\vartheta_0^2}{q}\Big)+(1+\tilde\vartheta_1)\frac{q}{\tilde t}-\frac{q(q\!+\!1)}{2(q\!-\!1)},
\end{equation*}
is obtained from $P_{VI}$ as a limit $\epsilon\to 0$ after the change of the independent variable
\begin{equation}\label{eq:HC-t}
t= 1+\epsilon\tilde t,\qquad \vartheta_t= \frac{1}{\epsilon},\quad \vartheta_1= -\frac{1}{\epsilon}+\tilde\vartheta_1+1,
\end{equation}
which sends the three singularities to $\tilde t=-\frac{1}{\epsilon},0,\infty$.
At the limit, the two simple singular points $-\frac{1}{\epsilon}$ and $\infty$ merge into a double (irregular) singularity at the infinity.%

The change of variables \eqref{eq:HC-t}, changes the function $\epsilon\cdot H_{VI}$ to
\begin{equation*}
\begin{split}
H_{VI,\epsilon}=\tfrac{1}{\tilde t(1+\epsilon\tilde t)}\Big[&
q(q-1)(q-1-\epsilon\tilde t)p^2-\Big(\vartheta_0(q-1)(q-1-\epsilon\tilde t)+\tilde\vartheta_1 q(q-1-\epsilon\tilde t)+\tilde tq-\epsilon\tilde tq \Big)p\\ 
& + \tfrac{(\vartheta_0+\tilde\vartheta_1)^2-\vartheta_\infty^2}{4}(q-1-\epsilon\tilde t)\Big],
\end{split}
\end{equation*}
and the Hamiltonian system to
\begin{align*}
\frac{dq}{d\tilde t}=\, \frac{\partial H_{VI,\epsilon}(q,p,\tilde t)}{\partial p},\qquad
\frac{dp}{d\tilde t}=-\frac{\partial H_{VI,\epsilon}(q,p,\tilde t)}{\partial q},
\end{align*}
whose limit $\epsilon\to 0$ is a Hamiltonian system of $P_V$.
In the coordinate $x=\frac{1}{\tilde t}+\epsilon$, the above system is written as 
\begin{equation}\label{eq:HC-PVIhamiltonianx}
\begin{aligned}
x(x-\epsilon)\frac{dq}{dx}=\, \frac{\partial H(q,p,x,\epsilon)\!}{\partial p},\qquad
x(x-\epsilon)\frac{dp}{dx}=-\frac{\partial H(q,p,x,\epsilon)\!}{\partial q},
\end{aligned}
\end{equation}
with
\begin{equation*}
\begin{aligned}
H(q,p,x,\epsilon)&= -(1+\epsilon\tilde t) H_{VI,\epsilon}(q,p,\tilde t)\\
&=-\tfrac{(\vartheta_0+\tilde\vartheta_1)^2-\vartheta_\infty^2}{4}((x-\epsilon)q-x)+x\vartheta_0p+\big(1-\epsilon-(x-\epsilon)\vartheta_0-x(\vartheta_0+\tilde\vartheta_1)\big)qp\\
&\qquad +(2x-\epsilon)(qp)^2 +(x-\epsilon)(\theta_0+\tilde\vartheta_1)q^2p-xqp^2-(x-\epsilon)q^3p^2,
\end{aligned}
\end{equation*}
and Theorem~\ref{theorem:HC-normalization} can be applied.

\begin{theorem}
The formal invariant $\chi$ of the system \eqref{eq:HC-PVIhamiltonianx} is
\begin{equation}
 \chi(h,x,\epsilon)=1-\epsilon-(x-\epsilon)\vartheta_0-x(\vartheta_0+\tilde\vartheta_1)+2(2x-\epsilon)h.
\end{equation}
\end{theorem}

\begin{proof}
Let 
\begin{alignat*}{2}
\tilde q&=q-xA,&\qquad  A&=\tfrac{\vartheta_0}{1+\epsilon(1+\vartheta_0+\tilde\vartheta_1)},\qquad\\
\tilde p&=p+(x-\epsilon)B,&  B&=\tfrac{(\vartheta_0+\tilde\vartheta_1)^2-\vartheta_\infty^2}{4(1+\epsilon(1-\vartheta_0))},
\end{alignat*}
and let 
\begin{align*}
\tilde H(\tilde q,\tilde p,x,\epsilon)=H(q,p,x,\epsilon)-\tfrac{(\vartheta_0+\tilde\vartheta_1)^2+\vartheta_\infty^2}{4t}x.
\end{align*}
Then for $\epsilon\neq 0$,
$$\tilde H(\tilde q,\tilde p,\epsilon,\epsilon)=
\big(1-\epsilon(1+\vartheta_0+\tilde\vartheta_1)\big)\tilde q\tilde p+\epsilon(\tilde q+\epsilon A)^2\tilde p^2-\epsilon(\tilde q+\epsilon A)\tilde p^2,$$
hence by Proposition~\ref{proposition:HC-lemma} the Birkhoff-Siegel invariant of $H(q,p,\epsilon,\epsilon)$ is
$$G(h,\epsilon,\epsilon)=\big(1-\epsilon(1+\vartheta_0+\tilde\vartheta_1)\big)h+\epsilon h^2,$$
and 
$$\tilde H(\tilde q,\tilde p,0,\epsilon)=
\big(1-\epsilon(1-\vartheta_0)\big)\tilde q\tilde p-\epsilon\tilde q^2(\tilde p+\epsilon B)^2
-\epsilon(\vartheta_0+\tilde\vartheta_1)\tilde q^2(\tilde p+\epsilon B) +\epsilon\tilde q^3(\tilde p+\epsilon B)^2,$$
hence by Proposition~\ref{proposition:HC-lemma} the Birkhoff-Siegel invariant of $H(q,p,0,\epsilon)$ is
$$G(h,0,\epsilon)=\big(1-\epsilon(1-\vartheta_0)\big)h-\epsilon h^2,$$
i.e.
$$G(h,x,\epsilon)=\big(1-\epsilon-(x-\epsilon)\vartheta_0-x(\vartheta_0+\tilde\vartheta_1)\big)h+ (2x-\epsilon)h^2.$$
\end{proof}

The Theorem~\ref{theorem:HC-normalization0} for the limit system $\epsilon=0$ is in this case due to Takano \cite{Ta}, see also \cite{Shi,Yo2}.
A separate paper \cite{Kl4} will be devoted to a more detailed study of the confluence $P_{VI}\to P_V$ and of the non-linear Stokes phenomenon in $P_V$ through the Riemann-Hilbert correspondence.

\section{Proof of Theorem~\ref{theorem:HC-normalization} and of Proposition~\ref{proposition:HC-lemma}}\label{sec:HC-proof}

The proof of Theorem~\ref{theorem:HC-normalization} is loosely based on the ideas of Siegel's proof of Theorem~\ref{theorem:HC-siegel} \cite[chap. 16 and 17]{SM}.
We construct  the normalizing transformation $y=\mbf\Phi_\pm(u,x,\epsilon)$ in a couple of steps as a formal power series in the $u$-variable with coefficients depending analytically on $(x,\epsilon)\in\sXE_\pm$,
and then show that the series is convergent. The main tool to prove the convergence is the Lemma~\ref{lemma:HC-majorantlemma} below.

\medskip

Let 
$$\phi_\pm(u,x,\epsilon)=\sum_{|\mbf m|\geq 2}\phi_{\pm,\mbf m}(x,\epsilon)u^{\mbf m},\qquad \mbf m=(m_1,m_2),\ u^{\mbf m}=u_1^{m_1}u_2^{m_2},\ |\mbf m|=m_1+m_2,$$
be a power series in the $u$-variable with coefficients bounded and analytic on $(x,\epsilon)\in\sXE_\pm$.
We will write
$$\{\phi_\pm\}_{\mbf m}:=\phi_{\pm,\mbf m}.$$
Denoting 
$$\|\phi_{\pm,\mbf m}\|:=\sup_{(x,\epsilon)\in\sXE_\pm} |\phi_{\pm,\mbf m}(x,\epsilon)|$$
the supremum norm over $\sXE_\pm$, let
$$\barbm\phi_\pm(u)=\sum_{|\mbf m|\geq 2}\|\phi_{\pm,\mbf m}\|u^{\mbf m},$$ 
be a majorant power series to $\phi$.
We will write 
$$\barbm \phi_\pm(u)\prec\barbm \psi_\pm(u)\quad\text{if}\quad \|\{\phi_\pm\}_{\mbf m}\|\leq \|\{\psi\}_{\mbf m}\|\ \text{ for all } \mbf m.$$

The following lemma is the essential technique in Siegel's proof.

\begin{lemma}\label{lemma:HC-majorantlemma}
Let $\phi=\transp{(\phi_1,\phi_2)}=O(u^2)$ be a formal power series in $u$, 
and let $r=\transp(r_1,r_2)=O(u^2)$ be a convergent power series in $u$.
If 
$$\barbm \phi_j(u)\prec \barbm r_j(u+\barbm\phi(u)),\qquad j=1,2,$$ 
where $\barbm\phi=\transp{(\barbm\phi_1,\barbm\phi_2)}$, then $\phi$ is convergent.
\end{lemma}

\begin{proof}
See \cite[Theorem 2.2]{RR11}. It can be also found implicitely in \cite[p.520]{MM} and \cite[chap. 17]{SM}.
\end{proof}

\subsection{Step 1: Ramified straightening of center manifold and diagonalization of the linear part}

\begin{itshape}
Suppose that the system is in a pre-normal form,
\begin{equation}\label{eq:HC-e1}
 x(x-\epsilon)\frac{dy}{dx}=J\transp{D_yF}(y,x,\epsilon),\qquad J\transp{D_yF}(y,x,\epsilon)=\chi(y_1y_2,x,\epsilon) \left(\begin{smallmatrix} 1 & 0 \\ 0 & -1 \end{smallmatrix}\right)y+O(x(x-\epsilon)).
\end{equation}
We will show that there exists a ramified transversely symplectic change of variable
\begin{equation}\label{eq:HC-T}
 y=T_\pm(x,\epsilon)w+\phi_{\pm,\mbf 0}(x,\epsilon),\qquad \det T_\pm(0,\epsilon)=I,\ \phi_{\pm,\mbf 0}(x,\epsilon)=O(x(x-\epsilon))
\end{equation}
bounded and analytic on the domain $\sXE_\pm$ of Definition~\ref{definition:HC-X},
that brings the system to a form
\begin{equation}\label{eq:HC-e4}
 x(x-\epsilon)\frac{dw}{dx}=\chi(w_1w_2,x,\epsilon) \left(\begin{smallmatrix} 1 & 0 \\ 0 & -1 \end{smallmatrix}\right)w+x(x-\epsilon) f_\pm(w,x,\epsilon),
\end{equation}
with $f_\pm(w,x,\epsilon)=O(|w|^2)$, $\frac{\partial f_{1,\pm}}{\partial w_1}+\frac{\partial f_{2,\pm}}{\partial w_2}=0$.
\end{itshape}

\medskip

The solution $w=0$ of the transformed system \eqref{eq:HC-e4}, 
corresponds to a bounded ramified solution $y=\phi_{\pm,\mbf 0}(x,\epsilon)$  of the system \eqref{eq:HC-e1}.
The paper \cite{Kl2}, see Theorem~\ref{proposition:HC-cm} below,
shows that there is a unique such solution on the domain $\sXE_\pm$; this it is the \emph{``ramified center manifold''} of the corresponding foliation.

The variable $\tilde y=y-\phi_{\pm,\mbf 0}(x,\epsilon)$ then satisfies
\begin{equation*}
 x(x-\epsilon)\frac{d\tilde y}{dx}=J\transp{D_yF}(\tilde y+\phi_{\pm,\mbf 0},x,\epsilon)-J\transp{D_yF}(\phi_{\pm,\mbf 0},x,\epsilon),
\end{equation*}
whose linear part is
$A_\pm(x,\epsilon):=J(D^2_{y}F)(\phi_{\pm,\mbf 0},x,\epsilon)=\lambda(x,\epsilon) \left(\begin{smallmatrix} 1 & 0 \\ 0 & -1 \end{smallmatrix}\right)+x(x-\epsilon)R(\phi_{\pm,\mbf 0},x,\epsilon).$
The transformation matrix $T_\pm$ \eqref{eq:HC-T} must then satisfy
\begin{equation*}
 x(x-\epsilon)\frac{d T_\pm}{dx}=A_\pm T_\pm-\lambda T_\pm \left(\begin{smallmatrix} 1 & 0 \\ 0 & -1 \end{smallmatrix}\right).
\end{equation*}
The existence of such a transformation $T_\pm$ bounded on $\sXE_\pm$ is known \cite{LR,HLR} when $A_\pm$ is analytic.
In our case the matrix $A_\pm$ is ramified, but their proof works anyway.
We will obtain $T_\pm$ directly using Theorem~\ref{proposition:HC-cm}.

Writing $R=\big(r_{ij}\big)_{i,j}$ and
$$T_\pm=\left(\begin{smallmatrix} 1 & t_{1,\pm} \\ t_{2,\pm} & 1 \end{smallmatrix}\right)\left(\begin{smallmatrix} e^{b_{1,\pm}} & 0 \\[3pt] 0 & e^{b_{2,\pm}} \end{smallmatrix}\right),$$
then the terms $t_{i,\pm}$, $i=1,2$, are solutions to Riccati equations
\begin{equation}\label{eq:HC-e6}
x(x-\epsilon)\tfrac{d t_{i,\pm}}{dx}=(-1)^{i-1}2\lambda t_{i,\pm}
+x(x-\epsilon)\big[r_{ij}+(r_{ii}-r_{jj})\cdot t_{i,\pm}-r_{ji}\cdot\big(t_{i,\pm}\big)^2\big],
\end{equation}
and the terms $b_{i,\pm}$ are solution to
$\tfrac{d b_{i,\pm}}{dx}=\big(r_{ii}+r_{ij} t_{j,\pm}\big),$ i.e.
$$b_{i,\pm}=\int_0^x r_{ii}+r_{ij}t_{j,\pm} dx.$$

Combining the equations \eqref{eq:HC-e1} for $\phi_{\pm,\mbf 0}$ and \eqref{eq:HC-e6} for $t_\pm$, in which $r_{ij}=r_{i,j}(\phi_{\pm,\mbf 0},x,\epsilon)$, 
we get an analytic system for which the existence of a unique bounded solution on $\sXE_\pm$ is assured by the following theorem.

\begin{theorem}[\hbox{\cite[Theorems 2 and 4]{Kl2}}]\label{proposition:HC-cm}
Consider a system of the form 
\begin{equation}\label{eq:HC-cm}
 x(x-\epsilon)\frac{d\phi}{dx}=M\phi + f(\phi,x,\epsilon), \qquad (\phi,x,\epsilon)\in \C^m\times\C\times\C.
\end{equation}
with $M$ an invertible $m\!\times\! m$-matrix whose eigenvalues are all\,\footnote{This assumption is inessential, it is added here just to simplify the statement. See \cite{Kl2} for a general version of the statement.}
on the line $\lambda^{(0)}\R$, and
$\,f(\phi, x,\epsilon)$ analytic germ such that $D_{\phi}f(0,0,0)=0$, and $f(0,x,\epsilon)=O(x(x-\epsilon))$.

\smallskip
\noindent
\textbf{(i)}
The system \eqref{eq:HC-cm} possesses a unique solution in terms of a formal power series in $(x,\epsilon)$:
\begin{equation}
 \hat \phi(x,\epsilon)=\sum_{k,j=0}^{+\infty}\phi_{kj}x^k\epsilon^j,\qquad \phi_{kj}\in\C^m.
\end{equation}
This series is divisible by $x(x-\epsilon)$, and its coefficients satisfy 
$\|\phi_{kj}\|\leq L^{k+j}(k+j)!$ for some $L>0$.

\smallskip
\noindent
\textbf{(ii)}
The system \eqref{eq:HC-cm} possesses a unique bounded analytic solution $\phi_\pm(x,\epsilon)$
on the domain $\sX_\pm(\epsilon)$, $\epsilon\in\sE_\pm$ of Definition~\ref{definition:HC-X} (for some $\delta_x,\delta_\epsilon>0$).
It is uniformly continuous on 
$$\sXE_\pm=\{(x,\epsilon)\mid x\in \sX_\pm(\epsilon)\}$$
and analytic on the interior of $\sXE_\pm$, and it vanishes (is uniformly $O(x(x-\epsilon))$) at the singular points. 
When $\epsilon$ tends radially to $0$ with $\arg\epsilon=\beta$, then 
$\phi(x,\epsilon)$ converges to $\phi(x,0)$ uniformly on compact sets of the sub-domains 
$\lim_{\substack{\epsilon\to 0\\ \arg\epsilon=\beta}}\sX_{\pm}(\epsilon)\subseteq \sX_\pm(0)$. 

\smallskip
\noindent
\textbf{(iii)}
Let $\Phi(x,\epsilon)$ be analytic extension of the function given by the convergent series
$$\Phi(x,\epsilon)=\sum_{j,k} \frac{\phi_{kj}}{(k+j)!}x^k\epsilon^j.$$ 
For each point $(x,\epsilon)$, for which there is $\theta \in\,]\!-\!\frac{\pi}{2},\frac{\pi}{2}[$ such that
$\,\mbox{$\mathbf{S}_\theta\cdot(x,\epsilon)$}\subseteq \sXE_\pm$, with $\mathbf{S}_\theta\subset\C$ denoting the circle  through the points $0$ and $1$ with  center on $e^{i\theta}\R^+$, we can express $\phi_\pm(x,\epsilon)$ as the 
following Laplace transform of $\Phi$:
\begin{equation}
\phi_\pm(x,\epsilon)=\int_0^{+\infty e^{i\theta}}\!\!\!\Phi(sx,s\epsilon)\,e^{-s}\,ds.
\end{equation}
In particular, $\phi_+(x,0)=\phi_-(x,0)$ is the functional cochain consisting of the pair of Borel sums of the formal series $\hat \phi(x,0)$ in directions on either side of $\lambda^{(0)}\R$.
\end{theorem}

Since the trace of the linear part of both systems \eqref{eq:HC-e1} and \eqref{eq:HC-e4} is null, then by the Liouville--Ostrogradskii formula $\det T(x,\epsilon)$ is constant in $x$ and equal to $\det T(0,\epsilon)=1$. Therefore the transformation \eqref{eq:HC-T} is transversely symplectic, and by Lemma~\ref{lemma:HC-transhamilton} the transformed system
\eqref{eq:HC-e4} is transversely Hamiltonian.

\subsection{Step 2: Normalization}

\begin{itshape}
Suppose that the system is in the form \eqref{eq:HC-e4}.
We will show that there exists a ramified change of variable $w=\Phi_\pm(v,x,\epsilon)$, $\Phi_\pm(\cdot,0,\epsilon)=\id$,
that will bring it to an integrable form
\begin{equation}\label{eq:HC-e44}
 x(x-\epsilon)\frac{dv}{dx}=\alpha_\pm(h,x,\epsilon) \left(\begin{smallmatrix} 1 & 0 \\[3pt] 0 & -1 \end{smallmatrix}\right)v,\qquad h=v_1v_2,
\end{equation}
for some germ $\alpha_\pm(h,x,\epsilon)$, with $\alpha_\pm(0,x,\epsilon)=\lambda(x,\epsilon)$.
\end{itshape}

\medskip
The transformation $\Phi_\pm$ must satisfy
\begin{equation}\label{eq:HC-eqPhi}
x(x-\epsilon)\partial_x\Phi_\pm+\alpha_\pm\cdot\!\left(v_1\partial_{v_1}\!-v_2\partial_{v_2}\right)\Phi_\pm=
\chi\circ\Phi_\pm\cdot \left(\begin{smallmatrix} 1 & 0 \\[3pt] 0 & -1 \end{smallmatrix}\right)\Phi_\pm+x(x-\epsilon) f_\pm\circ\Phi_\pm.
\end{equation}
We are looking for $\Phi_\pm$ written as
\begin{equation}\label{eq:HC-ePhi}
\Phi_\pm(v,x,\epsilon)=v+\Psi_{\pm,\Delta}+x(x-\epsilon)\Psi_{\pm,\star}(v,x,\epsilon),\qquad \Psi_{\pm,\Delta}+\Psi_{\pm,\star}=:\Psi_{\pm}
\end{equation}
where the power expansion of the $j$-th coordinate of $\Psi_{\pm,\Delta}$ is equal to
$$\Psi_{j,\pm,\Delta}=\sum_{n\geq 1}\{\Psi_{j,\pm}\}_{(n,n)+\mbf e_j}h^nv_j,\qquad \text{$\mbf e_j$ being the $j$-the elementary vector},$$
while the power expansion of the $j$-th coordinate of $\{\Psi_{\pm,\star}\}$ does not contains any power $h^nv_j$, $n\geq 0$.
In particular,
$\left(v_1\partial_{v_1}\!-v_2\partial_{v_2}\right)\Psi_{\pm,\Delta}=\left(\begin{smallmatrix} 1 & 0 \\[3pt] 0 & -1 \end{smallmatrix}\right)\Psi_{\pm,\Delta}$.
Therefore \eqref{eq:HC-eqPhi} becomes
\begin{equation}
\begin{aligned}
\partial_x \Psi_{\pm,\Delta}&+\partial_x\big(x(x-\epsilon)\Psi_{\pm,\star}\big)+
\lambda \left(v_1\partial_{v_1}\!-v_2\partial_{v_2}\!-\left(\begin{smallmatrix} 1 & 0 \\[3pt] 0 & -1 \end{smallmatrix}\right)\right)\Psi_{\pm,\star} \\
&=-\alpha_\pm^* \left(v_1\partial_{v_1}\!-v_2\partial_{v_2}\right)\Psi_{\pm,\star}+
\frac{\chi^*\circ\Phi_\pm-\alpha_\pm^*}{x(x-\epsilon)}\left(\begin{smallmatrix} 1 & 0 \\[3pt] 0 & -1 \end{smallmatrix}\right)(v+\Psi_{\pm,\Delta})
 + G_\pm,
\end{aligned}
\end{equation}
where
$G_\pm=\left(\chi^*\circ\Phi_\pm\right)\left(\begin{smallmatrix} 1 & 0 \\[3pt] 0 & -1 \end{smallmatrix}\right)\Psi_{\pm,\star}+f\circ\Phi_\pm,$
and 
$\chi^*=\chi-\lambda$, $\alpha_\pm^*=\alpha_\pm-\lambda$.

Set
\begin{equation*}
 \alpha_\pm(h,x,\epsilon)=\sum_{n\geq 0}\{\chi\circ\Phi_\pm\}_{(n,n)}h^n,
\end{equation*}
and denote
\begin{equation*}
 K_\pm:=\frac{\chi^*\circ\Phi_\pm-\chi^*\circ(v+\Psi_{\pm,\Delta})}{x(x-\epsilon)},
\end{equation*}
which is an analytic function of $v+\Psi_{\pm,\Delta}$ and $\Psi_{\pm,\star}$ with coefficients depending on $x,\epsilon$.
Then $\{\chi^*\circ\Phi_\pm-\alpha_\pm^*\}_{(n,n)}=0$ for all $n\geq 0$, and 
$\{\chi^*\circ\Phi_\pm-\alpha_\pm^*\}_{\mbf n}=x(x-\epsilon)\{K_\pm\}_{\mbf n}$ for all multi-indices $\mbf n$ with $n_1\neq n_2$,
since $\{\alpha_\pm^*\}_{\mbf n}=0=\{\chi^*\circ(v+\Psi_{\pm,\Delta})\}_{\mbf n}$.

Expanding the $j$-th coordinate, $j=1,2$, of the equation \eqref{eq:HC-eqPhi} in powers of $v$ we get:
\begin{itemize}
\item for $\mbf m=(n,n)+\mbf e_j$:
\begin{equation}\label{eq:HC-ed}
\partial_x\{\Psi_{j,\pm}\}_{(n,n)+\mbf e_j}=\{G_{j,\pm}\}_{(n,n)+\mbf e_j},
\end{equation}
\item for a multi-index $\mbf m$ with $m_1\!-m_2+(-1)^j\neq 0$:
\begin{equation}\label{eq:HC-em}
\begin{aligned}
&\partial_x \{x(x-\epsilon)\Psi_{j,\pm}\}_{\mbf m} + (m_1\!-m_2+(-1)^j) \lambda \{\Psi_{j,\pm}\}_{\mbf m}=\\
&-(m_1\!-m_2)\{\alpha_\pm^*\cdot\Psi_{j,\pm}\}_{\mbf m} - (-1)^j\{K_\pm\cdot(v_j+\Psi_{j,\pm,\Delta})\}_{\mbf m}  + \{G_{j,\pm}\}_{\mbf m}.
\end{aligned}
\end{equation}

\end{itemize}
The right sides of \eqref{eq:HC-em} and \eqref{eq:HC-ed} are functions of 
$\{\Psi_{\pm}\}_{\mbf k}=\transp(\{\Psi_{1,\pm}\}_{\mbf k},\{\Psi_{2,\pm}\}_{\mbf k})$, 
with $k_1\leq m_1,\ k_2\leq m_2$, $|\mbf k|<|\mbf m|$ only, which means that the equations for $\{\Psi_{j,\pm}\}_{\mbf m}$ can be solved recursively.

The equation \eqref{eq:HC-ed} is solved by
\begin{equation}\label{eq:HC-e8}
\{\Psi_{j,\pm}\}_{(n,n)+\mbf e_j}(x,\epsilon)=\int_0^x \{G_{j,\pm}\}_{(n,n)+\mbf e_j}dx.
\end{equation}
The equation \eqref{eq:HC-em} has a unique bounded solution $\{\Psi_{j,\pm}\}_{\mbf m}$ given by the integral
\begin{equation}\label{eq:HC-e9}
 \{\Psi_{j,\pm}\}_{\mbf m}(x,\epsilon)=\frac{e^{-(m_1-m_2+(-1)^j)t_\lambda}}{x(x-\epsilon)} \int_{x_{i,\pm}}^x e^{(m_1-m_2+(-1)^j)t_\lambda} 
\{F_{j,\pm}\}_{\mbf m}dx,
\end{equation}
where $\{F_{j,\pm}\}_{\mbf m}$ is the right side of \eqref{eq:HC-em},
\begin{equation*}
t_\lambda(x,\epsilon)=
\left\{\!\!\begin{array}{ll}   
-\tfrac{\lambda^{(0)}}{\epsilon}\log x+(\tfrac{\lambda^{(0)}}{\epsilon}+\lambda^{(1)})\log(x-\epsilon), & \text{for $\epsilon\neq 0$},\\[6pt]
-\frac{\lambda^{(0)}}{x}+\lambda^{(1)}\log x,& \text{for $\epsilon=0$,}
\end{array}\right.
\end{equation*}
is a branch of the rectifying coordinate for the vector field
\begin{equation*}
 \frac{x(x-\epsilon)}{\lambda(x,\epsilon)}\partial_x=\partial_{t_\lambda},
\end{equation*}
on $\sX_{\pm}(\epsilon)$,
and the integration follows a real trajectory of the vector field \eqref{eq:HC-xvectfield} in $\sX_{\pm}(\epsilon)$ from a point 
$$x_{*}= \left\{\!\!\begin{array}{ll}   
x_{1,\pm}, & \text{if $m_1-m_2+(-1)^j>0$},\\[6pt]
x_{2,\pm}, & \text{if $m_1-m_2+(-1)^j<0$,}
\end{array}\right.\qquad \text{$x_{i,\pm}$ is as in\eqref{eq:HC-x12}},$$
to $x$, along which the integral is well defined. 


Note that the convergence of the constructed formal transformation $\Phi_\pm$ is equivalent to the convergence of $\Psi_\pm$ \eqref{eq:HC-ePhi}.
We prove the convergence of the latter series using Lemma~\ref{lemma:HC-majorantlemma}.
For this we need to estimate the norms of \eqref{eq:HC-e8} and \eqref{eq:HC-e9}.

\begin{lemma}\label{lemma:HC-lemmastep3}
Let $\{\Psi_{j,\pm}\}_{\mbf m}$ be given by \eqref{eq:HC-e9}.
Then 
\begin{equation*}
 \|\{\Psi_{j,\pm}\}_{\mbf m}\| \leq \tfrac{c_\star}{m_1-m_2+(-1)^j}\|\{F_{j,\pm}\}_{\mbf m}\|,
\end{equation*}
for some $c_\star>0$ independent of $\mbf m$.
\end{lemma}

\begin{proof}
Let $\tau_k(x,\epsilon)=kt_\lambda(x,\epsilon)+\log(x(x-\epsilon))$.  
If $k=m_1-m_2+(-1)^j\neq 0$, then for $\epsilon$ small enough the integrating path can be deformed so that 
it corresponds to a ray $\tau_k(\xi,\epsilon)\in \tau_k(x,\epsilon)-e^{i\omega_\pm}[0,+\infty[$, with $\omega_\pm$ as in \eqref{eq:HC-omega}.
We have $\frac{d\tau_k}{dx}(x,\epsilon)=\frac{k\lambda(x,\epsilon)+2x-\epsilon}{x(x-\epsilon)}$. 
Hence 
$$\|\{\Psi_{j,\pm}\}_{\mbf m}\|\leq \frac{1}{\cos\omega_\pm} \left\| \frac{\{F_{j,\pm}\}_{\mbf m}(x,\epsilon)}{k\lambda(x,\epsilon)+2x-\epsilon}\right\|,$$
with $\cos\omega_\pm>\sin\eta>0$, $\eta$ as in \eqref{eq:HC-omega}.
\end{proof}

Therefore
\begin{itemize}
\item for $\mbf m=(n,n)+\mbf e_j$:
\begin{equation*}
\|\{\Psi_{j,\pm}\}_{(n,n)+\mbf e_j}\|\leq c\|\{G_{j,\pm}\}_{(n,n)+\mbf e_j}\|,
\end{equation*}
\item for a multi-index $\mbf m$ with $m_1\!-m_2+(-1)^j\neq 0$:
\begin{equation*}
\|\{\Psi_{j,\pm}\}_{\mbf m}\|\leq c \left(\|\{\alpha_\pm^*\Psi_{j,\pm}\}_{\mbf m}\| +\|\{K_\pm\cdot(v_j+\Psi_{j,\pm,\Delta})\}_{\mbf m}\| + \|\{G_{j,\pm}\}_{\mbf m}\|\right) ,
\end{equation*}
for some $c>0$.
\end{itemize}
Therefore
\begin{align*}
\barbm\Psi_{j,\pm} 
&\prec c\,\barbm\alpha_\pm^*\barbm\Psi_{j,\pm} +c\,\barbm K_\pm\cdot(v_j+\barbm\Psi_{j,\pm}) +c\,\barbm G_{j,\pm}\\
&\prec 2c\,(\barbm\chi^*\circ\barbm\Phi_\pm)\cdot \barbm\Psi_{j,\pm} +c\,\barbm k_\pm\cdot(v_j+\barbm\Psi_{j,\pm})  +c\,\barbm f_{j,\pm}\circ\barbm\Phi_\pm
=:\barbm R_{j,\pm}(v+\barbm\Psi_{\pm,\Delta},\barbm\Psi_{\pm,\star})\\
&\prec\barbm R_{j,\pm}(v+\barbm\Psi_{\pm},v+\barbm\Psi_{\pm}),
\end{align*}
where $K_\pm=K_\pm(v+\Psi_{\pm,\Delta},\Psi_{\pm,\star})$, $\barbm k_\pm=\barbm K_\pm(v+\barbm\Psi_{\pm,\Delta},\barbm\Psi_{\pm,\star})$,
and $\barbm R_{j,\pm}(w_1,w_2)=O(|w|^2)$, and we can conclude with Lemma~\ref{lemma:HC-majorantlemma}.

\subsection{Step 3: Final reduction and transverse symplecticity of the transformation}

Suppose that the system is in the form \eqref{eq:HC-e44}, and write 
$$\alpha_\pm(h,x,\epsilon)=\tilde\chi_\pm(h,x,\epsilon)+x(x-\epsilon)\beta_\pm(h,x,\epsilon),\quad
\tilde\chi_\pm(h,x,\epsilon)=\tilde\chi_\pm^{(0)}(h,\epsilon)+x\tilde\chi_\pm^{(1)}(h,\epsilon).$$
Then the transformation
$v=e^{\int_0^x\beta_\pm dx \left(\begin{smallmatrix} 1 & 0 \\ 0 & -1 \end{smallmatrix}\right)}u$
will bring it to the normal form with formal invariant $\tilde\chi_\pm$.

Let us show that the transformation 
$y=\mbf\Psi_\pm(u,x,\epsilon)$ obtained as a composition of the transformations of Steps 1-3
is transversely symplectic and therefore $\tilde\chi_\pm=\chi$.  

Let $u_\pm(x,\epsilon;c)$ be a germ of a general solution of the normal form system with the formal invariant equal to $\tilde\chi_\pm$,
depending on an initial condition parameter $c=(c_1,c_2)\in(\C^2,0)$, $\det(D_cu_\pm)\neq0$,
and let $y_\pm(x,\epsilon;c)=\mbf\Psi_\pm(u_\pm(x,\epsilon;c),x,\epsilon)$ be the corresponding solution germ of the system \eqref{eq:HC-system}.
Then $D_cy_\pm=D_u\mbf\Psi_\pm\cdot D_c u_\pm$ satisfies the linearized system
$$x(x-\epsilon)\frac{dD_cy_\pm}{dx}=JD_y^2H\cdot D_cy_\pm,$$
and by the Liouville-Ostrogradskii formula
$$x(x-\epsilon)\frac{d}{dx}\det(D_cy_\pm)=\tr(JD_y^2H)\cdot \det(D_cy_\pm),$$
but $\tr(JD_y^2H)=0$, i.e. $\det(D_cy_\pm)=\det(D_u\mbf\Psi_\pm)\cdot\det(D_c u_\pm)$ is constant in $x$. 
Similarly, $\det(D_cu_\pm)$ is constant in $x$. 
Therefore $\det(D_u\mbf\Psi_\pm(u,x,\epsilon))$ is also constant in $x$, and equal to $\det(D_u\mbf\Psi_\pm(u,0,\epsilon))=1$ since $\mbf\Psi_\pm(u,0,\epsilon)=u$.

\medskip
This terminates the proof of Theorem~\ref{theorem:HC-normalization}.

\subsection{Proof of Proposition~\ref{proposition:HC-lemma}}\label{section:HC-prooflemma}

We will construct a formal symplectic change of coordinate $\Phi=\Phi(h,u_i)$, written as a formal power series in $h$ and $u_i$, such that
$G=H\circ\Phi.$
The transformation $\Phi$ is constructed recursively as a formal limit
$$\Phi=\lim_{k\to+\infty}\Phi_{k,1},\quad
\Phi_{k+1,1}=\lim_{l\to+\infty}\Phi_{k,l},\quad
\Phi_{0,1}=\id,\quad
H\circ\Phi_{k,l}=G+O(h^ku_i^{l}).$$
At each step $(k,l)$, $k\geq 0$, $l\geq 1$, we want to get rid of the power $h^ku_i^l$ in $H\circ\Phi_{k,l}$.
We construct 
$\Phi_{k,l+1}=\Phi_{k,l}\circ\Phi^1_{f_{k,l}X_{h^ku_i^l}}$
as a composition of $\Phi_{k,l}$ with the time-1 flow  of a Hamiltonian vector field 
$f_{k,l}X_{h^ku_i^l}=-(-1)^if_{k,l}h^{k-1}u_i^l[ku_i\partial_{u_i}-(k+l)u_j\partial_{u_j}]$ for some $f_{k,l}\in\C$. 
The flow sends both $h$ and $u_i$ to functions of $(h,u_i)$,
$$\Phi^1_{f_{k,l}X_{h^ku_i^l}}:\quad u_i\mapsto u_i+O(h^{k-1}u_i^{l+1}),\quad h\mapsto h+(-1)^ilf_{k,l}h^ku_i^l+O(h^{2k-1}u_i^{2l}),$$
where the terms $O(h^{2k-1}u_i^{2l})$ are null if $k=0$. 
If $H\circ\Phi_{k,l}=G+H_{k,l}h^ku_i^l+O(h^ku_i^{l+1})$ for some $H_{k,l}\in\C$,
then we want 
\begin{align*}
(G+H_{k,l}h^ku_i^l)\circ\Phi^1_{f_{k,l}X_{h^ku_i^l}}-G(h)&=O(h^ku_i^{l+1}),\\
H_{k,l}h^ku_i^l+\lambda\cdot (-1)^ilf_{k,l}h^ku_i^l &=O(h^ku_i^{l+1}),\\
f_{k,l}&=-(-1)^i\tfrac{H_{k,l}}{\lambda l}.
\end{align*}

Now that we have constructed the formal symplectic transformation $\Phi$,
we can conclude by the following Proposition.

\begin{proposition}
Let  $H,\tilde H:(\C^2,0)\to(\C,0)$ be two germs with a non-degenerate critical point at $0$, and let $\omega,\tilde\omega$ be germs of symplectic forms.
Then the two pairs $(H,\omega)$, $(\tilde H,\tilde\omega)$ are analytically equivalent if and only if they are formally equivalent.
\end{proposition}

\begin{proof}
 By Theorem~\ref{theorem:HC-siegel}, the Birkhoff-Siegel normal form is, up to the involution \eqref{eq:HC-reflection}, a complete analytic invariant for each pair. Therefore it is enough to show that it is also a formal invariant. This can be seen from the invariance of a formalization of the formula \eqref{eq:HC-periodmap} of Section~\ref{section:HC-periodmap}.
\end{proof}

\bigskip
\goodbreak

\end{document}